\documentclass[a4paper]{article}

\usepackage {amsmath, amssymb, amsfonts, amsthm, mathabx, mathtools}
\usepackage{amscd}
\usepackage[english]{babel}
\usepackage{calc}
\usepackage{color}
\usepackage{dsfont}
\usepackage{enumerate}
\usepackage{epsfig}
\usepackage{geometry}
\usepackage{graphics}
\usepackage{graphicx}
\usepackage[utf8]{inputenc}
\usepackage{latexsym}
\usepackage{pdfpages}
\usepackage{pgfplots}
\usepackage{pstricks}
\usepackage{pxfonts}
 \usepackage[active]{srcltx}
\usepackage{tikz}
\usepackage{esint}
\usepackage{enumitem}

\usepackage[colorlinks=true, linkcolor=blue,citecolor=blue, urlcolor=blue]{hyperref}

\newcommand{\tn}[1]{\textnormal{#1}}
\renewcommand{\d}{\ \mathrm{d}}
\renewcommand{\div}{\mathrm{div}}
\newcommand{\card}{\mathrm{card}}

\newcommand{\N}{\mathbb{N}}
\newcommand{\R}{\mathbb{R}}
\newcommand{\RCupInf}{\R \cup \{\infty\}}
\newcommand{\RCupMInf}{\R \cup \{-\infty\}}

\newcommand{\SetNodes}{\mathcal{X}}
\newcommand{\SetEdges}{\mathcal{S}}

\newcommand{\SetKPre}{\mathcal{K}_{\tn{pre}}}
\newcommand{\SetK}{\mathcal{K}}
\newcommand{\SetKSmall}{K}
\newcommand{\SetJRho}{{\mathcal{J}_{\pm}}}
\newcommand{\SetJavg}{{\mathcal{J}_{avg}}}
\newcommand{\SetJavgMod}{{\hat{\mathcal{J}}_{avg}}}

\newcommand{\SetJequal}{{\mathcal{J}_{=}}}

\newcommand{\odens}{\mathbb{I}} 

\newcommand{\CE}{\mathcal{CE}}
\newcommand{\W}{\mathcal{W}}
\newcommand{\Action}{\mathcal{A}}
\newcommand{\F}{\mathcal{F}}
\newcommand{\G}{\mathcal{G}}
\newcommand{\EnergyFunctional}{\mathcal{E}}
\newcommand{\auxilaryFctAction}{\alpha}
\newcommand{\ActionU}{\widehat{\mathcal{A}}}

\newcommand{\prox}{\mathrm{prox}}
\newcommand{\proj}{\mathrm{proj}}
\newcommand{\indicatorFct}{{\mathcal{I}}}
\newcommand{\ol}[1]{\overline{#1}}

\DeclareMathOperator{\avg}{avg}
\DeclareMathOperator{\lindiscr}{\mathcal{I}} 

\newcommand{\InProdV}[2]{\langle #1,#2 \rangle_\pi}
\newcommand{\InProdE}[2]{\langle #1,#2 \rangle_Q}
\newcommand{\InProdH}[2]{\langle #1,#2 \rangle_H}

\newcommand{\NormV}[1]{\| #1 \|_\pi}
\newcommand{\NormE}[1]{\| #1 \|_Q}
\newcommand{\NormH}[1]{\| #1 \|_H}
\DeclareMathOperator{\id}{id}

\newcommand{\pr}{{\tn{pr}}}
\newcommand{\massNodes}{\rho}

\newcommand{\massEdges}{\vartheta}

\newcommand{\massVX}{\massNodes^{-}}
\newcommand{\massVY}{\massNodes^{+}}
\newcommand{\momentum}{m}

\newcommand{\LagrangeMultiplierCE}{\varphi}

\newcommand{\massElement}{\bar \rho}
\newcommand{\massElementSlack}{q}

\newcommand{\transportCostFunction}{\Phi}

\newcommand{\FESpaceZeroNodes}{V_{n,h}^0}
\newcommand{\FESpaceZeroEdges}{V_{e,h}^0}
\newcommand{\FESpaceOneNodes}{V_{n,h}^{1}}

\newcommand{\interval}[1]{I_{{#1}}}

\newcommand{\pProj}{p^{\pr}}
\newcommand{\nCone}{N}
\newcommand{\supdiff}{\partial^+}

\newcommand{\Ac}{\mathcal{A}}
\newcommand{\Bc}{\mathcal{B}}

\newcommand{\Mc}{\mathcal{M}}

\newcommand{\Pc}{\mathcal{P}}

\newcommand{\Wc}{\mathcal{W}}

\newcommand{\Entropy}{\mathcal{H}}

\newcommand{\weaklyStar}{\stackrel{*}{\rightharpoonup}}

\DeclareMathOperator\arctanh{arctanh}

\newtheorem{Definition}{Definition}[section]

\theoremstyle{remark}
\newtheorem{Remark}[Definition]{Remark}
\theoremstyle{plain}
\newtheorem{Proposition}[Definition]{Proposition}
\newtheorem{Lemma}[Definition]{Lemma}
\newtheorem{Theorem}[Definition]{Theorem}
\newtheorem{Corollary}[Definition]{Corollary}

\def\ie{\emph{i.e. }}

\DeclareMathOperator*{\argmin}{arg\,min}

\DeclareMathOperator{\inter}{int}

\newcommand{\iterl}[1]{#1^{(\ell)}}
\newcommand{\iterll}[1]{#1^{(\ell+1)}}

\newcommand{\iterz}[1]{#1^{(0)}}

\usepackage{algpseudocode}

\algrenewcommand{\algorithmiccomment}[1]{\hfill$//$ #1}

\newcommand{\LineIf}[1]{\State \textbf{if} #1}

\pgfplotsset{compat=newest}

\title{Computation of Optimal Transport on \\ Discrete Metric Measure Spaces}
\author{Matthias Erbar \footnote{Institute for Applied Mathematics, University of Bonn},
Martin Rumpf \footnote{Institute for Numerical Simulation, University of Bonn},
Bernhard Schmitzer \footnote{Institute for Applied Mathematics, University of M\"unster},
Stefan Simon\footnote{Institute for Numerical Simulation, University of Bonn}}

\date{\today}

\begin{document}

\maketitle

\begin{abstract}
  In this paper we investigate the numerical approximation of
    an analogue of the Wasserstein distance for optimal transport on
    graphs that is defined via a discrete modification of the
    Benamou--Brenier formula. This approach involves the logarithmic
  mean of measure densities on adjacent nodes of the graph.  For this
  model a variational time discretization of the probability densities
  on graph nodes and the momenta on graph edges is proposed.  A robust
  descent algorithm for the action functional is derived, which in
  particular uses a proximal splitting with an edgewise nonlinear
  projection on the convex subgraph of the logarithmic mean.  Thereby,
  suitable chosen slack variables avoid a global coupling of
  probability densities on all graph nodes in the projection step.
  For the time discrete action functional $\Gamma$--convergence to the
  time continuous action is established.  Numerical results for a
  selection of test cases show qualitative and quantitative properties
  of the optimal transport on graphs.  Finally, we use our
    algorithm to implement a JKO scheme for the gradient flow of the
    entropy in the discrete transportation distance, which is known to
    coincide with the underlying Markov semigroup, and test our results
    against a classical backward Euler discretization of this discrete
    heat flow.
\end{abstract}

\begin{center}
\end{center}

\section{Introduction}
\label{sec:intro}
For a metric space $(\SetNodes,d)$ and a weighting exponent
$p \in [1,\infty)$ optimal transport induces the $p$-Wasserstein
distances $W_p$ on the probability measures over $\SetNodes$.  A
remarkable property of Wasserstein distances is that they form a
length space if the base space $(\SetNodes,d)$ is a length space,
inducing the so-called displacement interpolation between probability
measures \cite{McCannConvexity1997}. The celebrated Benamou--Brenier
formula for $W_2$ over $\R^n$ \cite{BeBr00} can be interpreted as an
explicit search for the shortest path between two probability
measures. In the last two decades the geometry of metric spaces has
extensively been studied by means of optimal transport. In explicit it
has been observed that the 2-Wasserstein metric over probability
densities in $\R^n$ formally resembles a Riemannian manifold
\cite{Ot01} and that various diffusion-type equations can be
interpreted as gradient flows for entropy-type functionals with
respect to this metric \cite{JKO1998}. For a comprehensive
introduction we refer to the monographs
\cite{Villani-OptimalTransport-09,Santambrogio-OTAM}.

Unfortunately, this rich geometry is not directly available when the
base space $\SetNodes$ is discrete, since $W_2$ degenerates and does
not admit geodesics.
Maas \cite{Ma11} introduced a transport-type Riemannian metric $\Wc$
on probability measures over a discrete space $\SetNodes$ equipped
with a reversible Markov kernel $Q$, based on an adaption of the
Benamou--Brenier formula.  A key ingredient in the construction is the
choice of a \emph{`mass averaging'} function $\theta$ that
interpolates the amount of mass on neighbouring graph vertices. For
the particular choice of $\theta$ being the logarithmic mean, the heat
equation (with respect to the underlying Markov kernel) arises as
gradient flow of the entropy with respect to this metric \cite{Ma11,
  Mielke-Nonlinearity-2011}, yielding a discrete analogue of Otto's
interpretation of diffusive PDEs, see also \cite{MaasDiscretePME2014}
for a generalization to non-linear evolution equations on discrete
spaces. In analogy to the Lott--Sturm--Villani theory the displacement
interpolation on graphs has been used to introduce a notion of Ricci
curvature lower bounds for discrete spaces equipped with Markov
kernels \cite{ErbarMaas-Ricci2012} that implies a variety of
functional inequalities in analogy to the theory of
Lott--Sturm--Villani. The study of transport-type distances on
discrete domains has various connections to the original Wasserstein
distances on continuous domains.  Approximating a torus with an
increasingly finer toroidal graph, the discrete transport metric $\Wc$
has been shown to converge to the continuous underlying 2-Wasserstein
distance on the torus in the sense of Gromov--Hausdorff \cite{GiMa13}.
Conversely, the introduction of a mass averaging function for discrete
spaces has in turn inspired the design of new non-local transport-type
metrics in continuous domains \cite{Erbar-Jump2014}.
\medskip

Computing classical Wasserstein distances $W_2$ numerically is often a
challenge.  While the classical Kantorovich formulation via transport
couplings is a standard linear program, its na\"ive dense form
requires $(\card \SetNodes)^2$ variables which may quickly become
computationally unfeasible as $\SetNodes$ increases in size.  On
arbitrary metric graphs $(\SetNodes,d)$ an additional problem arises:
only local edge lengths are usually prescribed and the full distance
function $d : \SetNodes \times \SetNodes \to \R$ is in general unknown
a priori. On large graphs, computing $d$ from local edge lengths may
be computationally prohibitive or even storing $d$ may exceed the
memory capacities.

Owing to its particular structure, the 1-Wasserstein distance over a
discrete graph can be reformulated as a min cost flow problem along
its edges, thus drastically reducing the number of required variables
if the graph is sparse and requiring no pre-computation of $d$, see
for instance \cite{NetworkFlows1993}. On continuous domains this
corresponds to \emph{Beckmann's problem} \cite{Santambrogio-OTAM}.  A
numerical scheme tailored to application on meshed surfaces is
presented in \cite{SolomonEMDSurfaces2014}.  A computational approach
that uses quadratic regularization to break the non-uniqueness of the
optimal flow is described in \cite{EssidSolomon-QuadraticRegOT-2017}.

For the 2-Wasserstein distance on continuous domains the
Benamou--Brenier formula serves a similar purpose, see for instance
\cite{PaPe14} for a numerical scheme based on proximal point
algorithms. However, this does not immediately carry over to discrete
graphs, as the mass averaging function $\theta$ introduces a
non-trivial coupling of the mass variables along graph edges.  In
\cite{SolomonTransportGraphs2016} a Benamou--Brenier-type transport
distance on discrete metric graphs is developed, similar to the
construction of Maas, and a corresponding numerical scheme is
developed. A crucial design choice is that $\theta$ is picked to be
the harmonic mean which allows the application of second-order convex
cone programs for numerical optimization.  This does not extend to
other choices of $\theta$ and thus, for instance, hinders the
numerical study of the gradient flow when $\theta$ is the logarithmic
mean.

\subsubsection*{Contribution}  In this article we present a scheme
for the numerical approximation of the distance $\Wc$ on discrete sets
$\SetNodes$ equipped with irreducible Markov kernels $Q$ as introduced
by Maas.  We pick up the Benamou--Brenier-type formulation and provide
a temporal discretization of the action functional to obtain a
finite-dimensional convex problem and prove $\Gamma$-convergence of
the discretized functional to the original problem, as well as strong
convergence of the discrete geodesics to the continuous geodesics.  To
overcome the strong coupling of mass variables along graph edges
caused by the mass averaging function we introduce a set of slack
variables to remedy this entanglement.  This allows us to apply a
robust proximal point algorithm for the optimization. Due to the slack
variables, all involved proximal mappings can be computed efficiently
by either solving a sparse linear program (if $Q$ is sparse) or by
decomposing them into independent low-dimensional sub-problems.

In particular this numerical scheme does not depend critically on the
choice of $\theta$ and can be quickly adapted to different
variants. We provide formulas for the logarithmic and geometric mean.
For a series of numerical test cases we visualize and discuss the
behaviour of the interpolating flow.
Finally, we adopt the algorithm to approximate gradient flows with
respect to the discrete transportation distance $\Wc$. In particular,
we test the algorithm against a classical backward Euler
discretization of the heat equation on a graph which coincides with
the gradient flow of the entropy.

\subsubsection*{Organization} The paper is organized as follows. At first we review the
construction of the $L^2$-Wasserstein metric on discrete spaces by
Maas \cite{Ma11} in Section \ref{sec:review}.  Then, in Section
\ref{sec:Discretization} we will derive the time discretization and
establish $\Gamma$-convergence of the time discrete action functional
and the convergence of time discrete geodesics to a continuous
geodesic.  Next, the proximal splitting algorithm with suitably chosen
slack variables is presented in detail in Section
\ref{sec:ProximalSplitting}.  Numerical results are discussed in
Section \ref{sec:numerics} and the experimental comparison of
solutions of a JKO scheme for the entropy and solutions of the Markov
semigroup are presented in Section \ref{sec:gradflow}.

\section{Optimal transport on graphs}
\label{sec:review}
In this section we briefly review the discrete transportation metric
on the space of probability measures over a graph and in particular
recall the basic definitions and discuss the analogy to the
$L^2$-Wasserstein metric on probability measures over $\R^n$. Then we
derive a priori bounds on feasible curves of measures.

\subsection{The discrete transportation distance}
Let $\SetNodes$ be a finite set and let
$Q : \SetNodes \times \SetNodes \to [0,\infty)$ be the transition rate
matrix of a continuous time Markov chain on $\SetNodes$.
I.e.~we have $Q(x,y)\geq0$ for
$x\neq y$ and make the convention that $Q(x,x)=0$ for all
$x\in\SetNodes$. Then $\SetNodes$ can be interpreted as the set of
vertices of a graph with directed edges $(x,y)$ for those
$(x,y) \in \SetNodes \times \SetNodes$ with positive weight $Q(x,y)$.
We assume the Markov chain to be irreducible or equivalently the
corresponding graph to be strongly connected. Thus, there exists a
unique stationary distribution $\pi : \SetNodes \to (0,1]$ of the
Markov chain with $\sum_{x \in \SetNodes} \pi(x) = 1$. We further
assume that the Markov chain is reversible with respect to $\pi$,
i.e.~the detailed balance condition $\pi(x) Q(x,y) = \pi(y) Q(y,x)$
holds for all $x,y \in \SetNodes$.  Now, the set of probability
densities on $\SetNodes$ with respect to $\pi$ is given by
\begin{align*}
 \Pc(\SetNodes) := \left\{ \massNodes: \SetNodes \to \R^+_0 : \sum_{x \in \SetNodes} \pi(x) \massNodes(x) = 1 \right\} \, .
\end{align*}
For brevity, in the following we will write $\R^{\SetNodes}$ and $\R^{\SetNodes \times \SetNodes}$ for the spaces of real functions over $\SetNodes$ and $\SetNodes \times \SetNodes$ respectively.

Next, we define the following inner products on $\R^{\SetNodes}$ and $\R^{\SetNodes \times \SetNodes}$
\begin{align}
 \label{def:innerProducts}
 \InProdV{\phi}{\psi} & := \sum_{x \in \SetNodes} \phi(x) \psi(x) \pi(x), &
 \InProdE{\Phi}{\Psi} & := \frac{1}{2} \sum_{x,y \in \SetNodes} \Phi(x,y) \Psi(x,y) Q(x,y) \pi(x)
 \end{align}
 for $\phi$, $\psi \in \R^{\SetNodes}$ and $\Phi$, $\Psi \in \R^{\SetNodes \times \SetNodes}$.
The corresponding induced norms are denoted by $\NormV{\cdot}$ and $\NormE{\cdot}$.
A discrete gradient $\nabla_{\SetNodes} : \R^{\SetNodes} \to \R^{\SetNodes \times \SetNodes}$ and a discrete divergence 
$\div_{\SetNodes} : \R^{\SetNodes \times \SetNodes} \to \R^{\SetNodes}$ are given by
\begin{align}
 \label{def:graddiv}
 (\nabla_{\SetNodes} \psi)(x,y) & := \psi(x) - \psi(y), &
 (\div_{\SetNodes} \Psi)(x) & := \frac{1}{2} \sum_{y \in \SetNodes} Q(x,y) (\Psi(y,x) - \Psi(x,y)).
 \end{align}
Then the duality between these two operators formulated as  the discrete integration by parts formula
\begin{align*}
  \InProdV{\phi}{\div_{\SetNodes} \Psi} = -\InProdE{\nabla_\SetNodes \phi}{\Psi}
\end{align*}
can easily be verified.
The associated discrete Laplace-operator $\Delta_{\SetNodes} : \R^{\SetNodes} \to \R^{\SetNodes}$ is given by
\begin{align*}
 \Delta_{\SetNodes} \psi (x) 
 := \div_{\SetNodes} (\nabla_{\SetNodes} \psi) (x)
 = \sum_{y \in \SetNodes} Q(x,y)\left[ \psi(y)-\psi(x) \right]
 = (Q-D)\psi(x) \, ,
\end{align*}
where $D=\text{diag}(\sum_yQ(x,y))_{x\in\SetNodes}$.  The graph divergence allows
to formulate a continuity equation for time-dependent probability
densities $\massNodes : [0,1] \to \R^{\SetNodes}$ and momenta
$\momentum : [0,1] \to \R^{\SetNodes \times \SetNodes}$ describing the
flow of mass along the graph edges. In explicit, we consider the
following definition of solutions to the continuity equation with
boundary values at time $t=0$ and $t=1$.
\begin{Definition}[Continuity equation]\label{def:CEStrong}
The set $\CE(\massNodes_A, \massNodes_B)$ of solutions of the continuity equations for given 
boundary data $\massNodes_A$, $\massNodes_B \in \Pc(\SetNodes)$ is defined as the set of all pairs $(\massNodes,\momentum)$ 
with $\massNodes:[0,1] \times \R^\SetNodes \to \R$ and $\momentum:[0,1] \times \R^{\SetNodes \times \SetNodes} \to \R$ measurable, such that  
\begin{equation}
\int_0^1 \InProdV{\partial_t \LagrangeMultiplierCE(t,\cdot)}{\massNodes(t,\cdot)} +
 \InProdE{\nabla_\SetNodes \LagrangeMultiplierCE(t,\cdot)}{\momentum(t,\cdot)}   \d t= 
\InProdV{\LagrangeMultiplierCE(1,\cdot)}{\massNodes_B} - \InProdV{\LagrangeMultiplierCE(0,\cdot)}{\massNodes_A}
\label{eq:CEStrong}
\end{equation}
for all $\LagrangeMultiplierCE \in C^1([0,1],\R^\SetNodes)$. 
\end{Definition}

For $m\in L^2((0,1),\R^{\SetNodes \times \SetNodes})$ (see Lemma \ref{Lemma:L2}) 
one gets $\massNodes \in H^{1,2}((0,1),\R^{\SetNodes})$ and thus 
$\partial_t \massNodes + \div_\SetNodes \momentum =0$ holds a.e.~. 
Furthermore, $\massNodes \in C^{0,\frac12}([0,1],\R^{\SetNodes \times \SetNodes})$ and 
$\massNodes(1,\cdot) =  \massNodes_B$, $\massNodes(0,\cdot) =  \massNodes_A$.
If $\massNodes(t,\cdot) \geq 0$ is ensured for all $t\in (0,1)$ via a finite energy property (see \eqref{def:localAction} below), then 
testing with $\LagrangeMultiplierCE(t,x) = \zeta(t)$ implies that $\massNodes(t,\cdot) \in \Pc(\SetNodes)$.

The Benamou--Brenier formula \cite{BeBr00} asserts that the squared
$L^2$-Wasserstein distance for probability measures in $\R^n$ is the
minimum of an action functional over solutions to the corresponding
continuity equation. Formally the action functional can be interpreted
as a Riemannian path length \cite{Ot01}. To construct an analogous
action functional for solutions
$(\massNodes,\momentum) \in \CE(\massNodes_A,\massNodes_B)$ a mass
density on edges has to be deduced from the the mass densities on the
edge nodes. To this end, one defines an averaging function
$\theta :(\R^+_0)^2 \to \R^+_0$ which satisfies: \smallskip

$\theta$ is continuous, concave, 1-homogeneous, and symmetric, $\theta$ is $C^\infty$ on $(0,+\infty)^2$,
$\theta(0,s) = \theta(s,0) = 0$ and $\theta(s,s)=s$ for $s \in \R^+_0$, $\theta(s,t) > 0$ if $s>0$ and $t>0$, and
$s \mapsto \theta(t,s)$ is monotone increasing on $\R^+_0$ for fixed $t \in \R^+_0$.
\smallskip

It will be useful to consider $\theta$ as a concave function
$\R^2 \to \RCupMInf$. Therefore, we will set $\theta(s,t) = -\infty$
when $\min\{s,t\}<0$.  Possible choices for $\theta$ are for example
the logarithmic mean $\theta_{\tn{log}}$ or the geometric mean
$\theta_{\tn{geo}}$ for $s$, $t \in \R^+_0$:
\begin{align}
\label{def:means}
 \theta_{\tn{log}}(s,t) =
\begin{cases} 0, & \text{ if } s = 0 \text{ or } t = 0 \\
              s, & \text{ if } s=t \\
              \frac{t-s}{\log(t) - \log(s)} & \text{ otherwise }
\end{cases}
\quad  , \qquad \qquad
 \theta_{\tn{geo}}(s,t) = \sqrt{st} \, .
\end{align}
Note that the arithmetic mean is not admissible.  Based on this
averaging function one can define the discrete transportation distance
on $\Pc(\SetNodes)$.
\begin{Definition}[Action functional and distance]\label{defi:definition_of_metric}
The action functional for measurable functions $\massNodes : [0,1] \to \R^\SetNodes$
and $\momentum : [0,1] \to \R^{\SetNodes \times \SetNodes}$ is defined as
\begin{multline} 
	\label{def:localAction}
	\Action(\massNodes,\momentum) = \frac12 \int_0^1 \sum_{x,y \in \SetNodes} \auxilaryFctAction\big( \massNodes(t,x), \massNodes(t,y), \momentum(t,x,y) \big) \, Q(x,y) \, \pi(x) \d t
	\\
\tn{with } \auxilaryFctAction  : \R^3 \to \RCupInf;  \quad
 (s,t,m)  \mapsto 
 \begin{cases}
 \frac{m^2}{\theta(s,t)} & \tn{if } \theta(s,t) > 0, \\
 0 & \tn{if } \theta(s,t) = 0 \tn{ and } m = 0, \\
 + \infty & \tn{else.}
 \end{cases}
\end{multline}
The energy is then given by
\begin{align*}
	\EnergyFunctional(\massNodes, \momentum) &
	= \Action(\massNodes, \momentum)
	+ \indicatorFct_{\CE(\massNodes_A,\massNodes_B)}(\massNodes, \momentum)\,,
\end{align*}
where $\indicatorFct_{\CE(\massNodes_A,\massNodes_B)}$ is the
indicator functional, which is zero for $(\massNodes,\momentum)$ in
$\CE(\massNodes_A,\massNodes_B)$ and $\infty$ otherwise.  The induced
discrete transportation distance is obtained by
\begin{equation}\label{def:metric}
 \Wc(\massNodes_A, \massNodes_B) = \sqrt{\inf \EnergyFunctional(\massNodes, \momentum)}\,.
\end{equation}
\end{Definition}

Note that $\auxilaryFctAction$ is convex and lower semi-continuous and
$\CE(\massNodes_A,\massNodes_B)$ is a convex set. Hence,
\eqref{def:metric} is a convex optimization problem.  In is shown in
\cite[Theorem 3.8]{Ma11} that the mapping
$\Wc: \Pc(\SetNodes) \times \Pc(\SetNodes) \to \R$ defines a metric on
$\Pc(\SetNodes)$, provided
\begin{align*}
  \int_0^1\frac{1}{\sqrt{\theta(1-r,1+r)}}\d r <\infty\;.
\end{align*}
This is the case for the logarithmic mean $\theta_{\log}$ and the
geometric mean $\theta_{\rm geo}$. In \cite[Theorem
3.2]{ErbarMaas-Ricci2012} it is shown that the infimum in
\eqref{def:metric} is attained by an optimal pair $(\rho,\psi)$. The
curve $(\rho_t)_{t\in[0,1]}$ is a constant speed geodesic for the
distance $\Wc$, i.e. it holds
$\Wc(\rho_t,\rho_s)=|t-s|\Wc(\rho_A,\rho_B)$ for all $s,t\in[0,1]$.

\subsection{A priori bounds}
In what follows we will investigate the numerical approximation of
$\Wc$ using a suitable Galerkin discretization in time and solving the
resulting discrete convex optimization problem.  Here the the
nonlinear averaging function $\theta$ and the resulting coupling of
the values of the probability desity on neighbouring nodes will
require special treatment in order to obtain a robust and effective
solution scheme.
To this end, we first discuss a few simplifications of the optimization problem \eqref{def:metric} that will help to reduce the computational complexity.
\begin{Remark}[Sparsity of kernel $Q$]
\label{rem:Sparsity}
Let $\SetEdges = \{ (x,y) \in \SetNodes^2 \,\colon\, Q(x,y)>0 \}$ be the set of `edges' indicated by non-zero transition probability. 
As $Q$ is reversible, one finds $(x,y) \in \SetEdges$ iff $(y,x) \in \SetEdges$. 
Furthermore,  
$\div_\SetNodes \momentum(t,\cdot)$ and $\Action(\massNodes,\momentum)$ 
for $\momentum : [0,1] \to \R^{\SetNodes \times \SetNodes}$ only depend on values of $\momentum(t,x,y)$ where $(x,y) \in \SetEdges$.
Hence,if the kernel $Q$ is sparse, i.e.~if $\SetEdges$ is only a small subset of $\SetNodes \times \SetNodes$ this implies a considerable reduction of computational complexity.
\end{Remark}
In addition, the following Lemma allows to replace the two variables $\momentum(t,x,y)$ and $\momentum(t,y,x)$ by one effective variable, further reducing the problem size.
\begin{Lemma}[Antisymmetry of optimal
  momentum] \label{Lem:AntisymmetryM} If
  $\Wc(\massNodes_A,\massNodes_B)$ is finite and if
  $\massNodes:[0,1]\to\R^\SetNodes$ and
  $\momentum : [0,1] \to \R^{\SetNodes \times \SetNodes}$ are optimal
  for \eqref{def:metric} then $\momentum(t,x,y) = -\momentum(t,y,x)$
  $t$-almost everywhere, whenever $(x,y) \in \SetEdges$ (see above
  remark for definition of $\SetEdges$).
\end{Lemma}
\begin{proof}
Let $\massNodes : [0,1] \to \R^\SetNodes$ and $\momentum : [0,1] \to \R^{\SetNodes \times \SetNodes}$ be given such that $\EnergyFunctional(\massNodes,\momentum) < \infty$. Now set
\begin{align*}
\hat{\momentum}(t,x,y) := -\momentum(t,y,x)\,.
\end{align*}
One quickly verifies that $\div_\SetNodes \hat{\momentum} = \div_\SetNodes \momentum$ and that thus $(\massNodes,\hat{\momentum}) \in \CE(\massNodes_A,\massNodes_B)$ as well.
Besides, by using that $Q(x,y)\,\pi(x) = Q(y,x)\,\pi(y)$ and $\auxilaryFctAction(s,t,m)=\auxilaryFctAction(t,s,-m)$ one finds that $\Action(\massNodes,\hat{\momentum})=\Action(\massNodes,\momentum)$.
Let now $\ol{\momentum}=\frac12 (\momentum + \hat{\momentum})$. Note that $\ol{\momentum}(t,x,y)$ is anti-symmetric in $x$ and $y$. By convexity of $\CE(\massNodes_A,\massNodes_B)$ one gets $(\massNodes,\ol{\momentum}) \in \CE(\massNodes_A,\massNodes_B)$ and by convexity of $\Action$ one finds
\begin{align*}
\Action(\massNodes,\ol{\momentum}) \leq \frac12 \left(
\Action(\massNodes,\momentum) + \Action(\massNodes,\hat{\momentum})\right)
= \Action(\massNodes,\momentum)\,.
\end{align*}
Further, the finiteness of $\Action(\massNodes,\momentum)$ implies that $\momentum(t,x,y)=0$ when $\theta(\massNodes(t,x),\massNodes(t,y))=0$ 
and $(x,y) \in \SetEdges$ $t$-almost everywhere , values of $\momentum(t,x,y)$ for $(x,y) \notin \SetEdges$ will have no impact on $\Action$,
and the function $\R \ni z \mapsto \auxilaryFctAction(s,t,z)$ is even strictly convex for fixed $s,\,t > 0$.
Hence, we observe that $\Action(\massNodes,\ol{\momentum})<\Action(\massNodes,\momentum)$ unless $\ol{\momentum}$ already coincides with $\momentum$ for almost every $t$ and all $(x,y) \in \SetEdges$.
\end{proof}
In the $\Gamma$-convergence analysis we will make use on the following $L^2$ bound for the momentum.
Let us introduce the constants
\begin{align}\nonumber
  C^*:=&\max_{x\in\SetNodes} \sum_yQ(x,y)\;,\\\label{eq:constantsC}
  C_*:=&\min_{x,y\in\SetNodes, Q(x,y)>0}Q(x,y)\pi(x)\;.
\end{align}
\begin{Lemma}[$L^2$ bound for the momentum]\label{Lemma:L2}
Let $(\massNodes,\momentum): [0,1] \to \R^\SetNodes \times \R^{\SetNodes \times \SetNodes}$ 
be a measurable path with energy  $\EnergyFunctional(\massNodes,\momentum) \leq \widebar E<\infty$. Then, $\momentum$ and $\massNodes$ are uniformly bounded in 
$L^2((0,1), \R^{\SetNodes \times \SetNodes})$ and 
$H^{1,2}((0,1),\R^\SetNodes) \cap C^{0,\frac12}([0,1],\R^\SetNodes)$, respectively, with 
bounds solely depending on $\SetNodes$ and $\widebar E$.
\end{Lemma}
\begin{proof}
   Since $\EnergyFunctional(\massNodes,\momentum) < \infty$, 
   we have $(\massNodes,\momentum) \in \CE(\massNodes_A, \massNodes_B)$, 
   and thus for a.e.~$t \in (0,1)$ the mass is preserved, \ie  $\sum_{x \in \SetNodes} \massNodes(t,x) \pi(x) = \sum_{x \in \SetNodes} \massNodes_A(x) \pi(x) = 1$. 
   In addition, $\massNodes(t,x)$ is non-negative for all $x \in \SetNodes$ and a.e.~$t \in (0,1)$.
   By symmetry and concavity of $\theta$ and since $\theta(s,s)=s$, we can estimate 
   \begin{align*}
     \theta(\massNodes(t,x), \massNodes(t,y) ) &= \frac12\theta( \massNodes(t,x),\massNodes(t,y))+\frac12\theta(\massNodes(t,y),\massNodes(t,x)) \\
     &\leq \theta\Big(\frac{\massNodes(t,x)+\massNodes(t,y)}{2},\frac{\massNodes(t,x)+\massNodes(t,y)}{2}\Big) = \frac{\massNodes(t,x) + \massNodes(t,y)}{2}
   \end{align*}
   and get
  \begin{align}\nonumber 
   \sum_{x,y \in \SetNodes} \theta( \massNodes(t,x), \massNodes(t,y)) Q(x,y) \pi(x) 
   \leq \frac12\sum_{x,y \in \SetNodes} (\massNodes(t,x) Q(x,y) \pi(x) + \massNodes(t,y) Q(y,x) \pi(y) ) \\
=  \frac12\sum_{x,y \in \SetNodes} (\massNodes(t,x) Q(y,x) \pi(y) + \massNodes(t,y) Q(x,y) \pi(x) )   =   C^*\, \sum_{x \in \SetNodes} \massNodes(t,x) \pi(x) =  C^*\, . \label{eq:ThetaMonotonictyArg}
  \end{align}
  Thus, using the Cauchy--Schwarz inequality we obtain
  \begin{align}\nonumber
  \Big( \sum_{x,y \in \SetNodes} |\momentum(t,x,y)| Q(x,y) \pi(x) \Big)^2
  \leq & \Big( \sum_{x,y \in \SetNodes} \auxilaryFctAction( \massNodes(t,x),\massNodes(t,y), \momentum(t,x,y) ) Q(x,y) \pi(x) \Big) \\
  & \cdot\Big( \sum_{x,y \in \SetNodes} \theta(\massNodes(t,x), \massNodes(t,y) Q(x,y) \pi(x) \Big) \, . \label{eq:L2BoundM}
  \end{align}
Integrating in time we obtain 
  \begin{align*}
  \int_0^1||\momentum(t,\cdot,\cdot)||^2_Q\d t = \int_0^1 \sum_{x,y \in \SetNodes} \momentum(t,x,y)^2 Q(x,y) \pi(x) \d t \leq  \frac{C^*}{C_*} \, \widebar E \,.
  \end{align*}
Finally,  using the continuity equation \eqref{eq:CEStrong} and $\momentum$ in $L^2((0,1),\R^{\SetNodes \times \SetNodes})$
we obtain that 
\begin{align*}
  \int_0^1||\partial_t\massNodes||^2_\pi\d t \leq \int_0^1\sum_x\Big|\sum_y\momentum(t,x,y)Q(x,y)\Big|^2\pi(x)
d t \leq C^*\int_0^1\sum_{x,y}\momentum(t,x,y)^2Q(x,y)\pi(x)d t\;.
\end{align*}

This implies that $\massNodes \in H^{1,2}((0,1),\R^\SetNodes)$ and via
the Sobolev embedding theorem we obtain that also
$\massNodes \in C^{0,\frac12}((0,1),\R^\SetNodes)$.
\end{proof}

\section{Discretization}\label{sec:Discretization}
\subsection{Galerkin discretization}\label{sec:Galerkin}
To approximate the minimizers of \eqref{def:metric} numerically we
choose a Galerkin discretization in time.  The time interval $[0,1]$
is divided into $N$ subintervals $\interval{i}=[t_i,t_{i+1})$ for
$i=0, \ldots, N-1$ with uniform step size $h=\frac{1}{N}$ and
$t_i=i\,h$.  Then, we define discrete spaces
\begin{align*}
 \FESpaceOneNodes & = \{ \psi_h \in C^0([0,1],\R^{\SetNodes})  \,\colon\,
 	\psi_h(\cdot)|_{\interval{i}} \tn{ is affine } \forall i=0,\ldots,N-1
 	\} \, , \\
 \FESpaceZeroNodes & = \{ \psi_h : [0,1] \to \R^{\SetNodes} \,\colon\,
 	\psi_h(\cdot)|_{\interval{i}} \tn{ is constant } \forall i=0, \ldots, N-1
 	\} \, , \\
 \FESpaceZeroEdges & = \{ \psi_h : [0,1] \to \R^{\SetNodes \times \SetNodes} \, \colon \,
 	\psi_h(\cdot)|_{\interval{i}} \tn{ is constant } \forall i=0, \ldots, N-1 \}\,.
\end{align*}
For a function $\psi_h \in \FESpaceZeroNodes$ or $\FESpaceZeroEdges$ we will often write $\psi_h(t_i)$ to refer to its value on the interval $\interval{i}=[t_i,t_{i+1})$.
For a function $\psi_h \in \FESpaceOneNodes$ the time-derivative can be interpreted as map
\begin{align*}
	\partial_t : \FESpaceOneNodes & \to \FESpaceZeroNodes\;, &
	(\partial_t \psi_h)(t_i) & =\frac{1}{h}(\psi_h(t_{i+1})-\psi_h(t_i))
	\tn{ for } i=0,\ldots,N-1\,.
\end{align*}

We pick $\FESpaceOneNodes \times \FESpaceZeroEdges$ as the space for discretized masses and momenta $(\massNodes_h,\momentum_h)$. That is, discrete masses $\massNodes_h$ are continuous and piecewise affine and the corresponding momenta $\momentum_h$ will be piecewise constant. $\partial_t \massNodes_h$ and $\div_\SetNodes \momentum_h$ then lie in $\FESpaceZeroNodes$.
In analogy to Definition \ref{def:CEStrong} we define discrete solutions of the continuity equation.
\begin{Definition}
The set of solutions to the discretized continuity equation for given boundary values $\massNodes_A, \massNodes_B \in \R^\SetNodes$ is given by
 \begin{multline}\label{eq:CEDisc}
   \CE_h(\massNodes_A, \massNodes_B)
   = \Big\{ (\massNodes_h, \momentum_h) \in \FESpaceOneNodes \times \FESpaceZeroEdges \; : \; 
      h \sum_{i=0}^{N-1} \InProdV{\partial_t \massNodes_h(t_i,\cdot) + \div_\SetNodes \momentum_h(t_i,\cdot)}{\LagrangeMultiplierCE_h(t_i,\cdot)} = 0 \; \forall \LagrangeMultiplierCE_h \in \FESpaceZeroNodes \, , \\
      \massNodes_h(t_0, x) = \massNodes_A(x) \, , \massNodes_h(t_N, x) = \massNodes_B(x)
      \Big\} \, .
\end{multline}
\end{Definition}
One can quickly verify that
$\CE_h(\massNodes_A,\massNodes_B) = \CE(\massNodes_A,\massNodes_B)
\cap (\FESpaceOneNodes \times \FESpaceZeroEdges)$
and that $\partial_t \massNodes_h + \div_\SetNodes \momentum_h=0$
holds for a.e.~$t$ when
$(\massNodes_h,\momentum_h) \in \CE_h(\massNodes_A,\massNodes_B)$.
Next, we define a fully discrete action functional in analogy to
Definition~\ref{defi:definition_of_metric} and subsequently a discrete
version of the transport metric $\Wc$.

\begin{Definition}[Time-discrete action and transportation distance]
The averaging operator $\avg_h$ takes a measure $\psi \in \Mc([0,1],\R^{\SetNodes})$ to its average values on time intervals $\interval{i}$:
\begin{align*}
	\avg_h : \Mc([0,1],\R^{\SetNodes}) & \to \FESpaceZeroNodes\,, \quad
	(\avg_h \psi)(t_i)  = \psi(\interval{i})\;
	\tn{ for } i=0,\ldots,N-1\,.
\end{align*}
Analogously we declare the $\avg_h$ operator for $\R^{\SetNodes \times \SetNodes}$-valued measures. Note that for $\psi_h \in \FESpaceOneNodes$ one finds
$(\avg_h \psi_h)(t_i)=\tfrac12(\psi_h(t_i)+\psi_h(t_{i+1}))\,$.
For $(\massNodes,\momentum) \in \Mc([0,1],\R^{\SetNodes}) \times \Mc([0,1],\R^{\SetNodes \times \SetNodes}) $ the discrete approximation for the action is given by
\begin{align*}
	\Action_h(\massNodes, \momentum) & =
	\Action(\avg_h \massNodes, \avg_h \momentum) \\
	&=\frac{h}{2} \sum_{i=0}^{N-1} \sum_{x,y \in \SetNodes} \auxilaryFctAction \left( \avg_h \massNodes(t_i,x) , \avg_h \massNodes(t_i,y), \avg_h \momentum(t_i,x,y) \right) Q(x,y) \pi(x) \, .
\end{align*}
Finally, the time discrete energy functional is defined by
$ \EnergyFunctional_h(\massNodes, \momentum) =
\Action_h(\massNodes,\momentum) +
\indicatorFct_{\CE_h(\massNodes_A,\massNodes_B)}(\massNodes,
\momentum) $
and for the associated time discrete approximation of the
transportation distance one obtains
\begin{align}\label{def:metricDiscrete}
 \W_h(\massNodes_A, \massNodes_B) = \sqrt{\inf \EnergyFunctional_{h}(\massNodes, \momentum)}\,.
\end{align}
\end{Definition}

Note that the indicator function of the discrete continuity equation entails the constraint $(\massNodes,\momentum) \in \FESpaceOneNodes \times \FESpaceZeroEdges$.
These spaces can be represented by finite-dimensional vectors, the operators $\partial_t$ and $\avg_h$ can be represented as finite-dimensional matrices and the continuity equation becomes a finite-dimensional affine constraint. 
Thus, \eqref{def:metricDiscrete} is indeed a finite-dimensional convex optimization problem. Its numerical solution by using proximal mappings will be detailed in Section~\ref{sec:ProximalSplitting}.

\subsection[Gamma-convergence]{$\Gamma$-convergence}

In the following, we will prove a $\Gamma$-convergence result of the discrete energy functional, which will justify our discretization.
First, we construct explicitly continuous and discrete trajectories between an arbitrary probability distribution on $\SetNodes$ and the uniform probability density $\odens \in \Pc(\SetNodes)$ given by $\odens(x) = 1$. We show that these trajectories have uniformly bounded energy, which will be essential in the $\Gamma$-$\limsup$ inequality in Theorem~\ref{thm:GammaConvergence}. Let us define the 
Lagrange interpolation operator $\lindiscr_h: C^0([0,1],\R^\SetNodes) \to \FESpaceOneNodes; \rho \mapsto \lindiscr_h(\rho)$ given by
\begin{align*}
\left(\lindiscr_h \massNodes \right)(t_i,x) := \massNodes(t_i,x) \quad \forall i=0,\ldots,N \, .
\end{align*}

\begin{Proposition}
	\label{prop:FiniteCostToFlatMeasure}
	There is some constant $C(\SetNodes) < \infty$ such that for
        any $\massNodes_A \in \Pc(\SetNodes)$ there is a trajectory
        $(\massNodes,\momentum) \in \CE(\massNodes_A,\odens)$ with
        $\Action(\massNodes,\momentum) \leq C(\SetNodes)$ and
        $(\lindiscr_h\massNodes,\avg_h\momentum) \in
        \CE_h(\massNodes_A,\odens)$
        with
        $\Action_h(\lindiscr_h\massNodes,\avg_h\momentum)) \leq
        C(\SetNodes)$ for every $h=1/N$.
\end{Proposition}
\begin{proof}
For $x \in \SetNodes$ let $\rho^x_A\in \Pc(\SetNodes)$ be the probability density on $\SetNodes$ with all mass concentrated on $x$. That is, 
$\rho^x_A =  \tfrac1{\pi(x)} \delta_x$, where $\delta_x$ is the usual Kronecker symbol with $\delta_x(y) = 1$ if $x=y$ and $0$ else.
\medskip

\emph{Construction of elementary flows:}
For $(x,y) \in \SetNodes \times \SetNodes$, $x\neq y$, with $Q(x,y)>0$ we define $L[x,y] \in \R^{\SetNodes \times \SetNodes}$ as follows:
\begin{align*}
L[x,y](a,b) = \begin{cases}
\frac{1}{Q(x,y)\,\pi(x)} & \tn{if } (a,b) = (x,y), \\
\frac{-1}{Q(x,y)\,\pi(x)} & \tn{if } (a,b) = (y,x), \\
0 & \tn{else.}
\end{cases}
\end{align*}
Then $\div_\SetNodes L[x,y]=\rho^y_A -\rho^x_A$.
Now, for any $(x,y) \in \SetNodes \times \SetNodes$, $x \neq y$, there exists a path $(x=x_0,x_1,\ldots,x_K=y)$ with $K<\card\, \SetNodes$ with $Q(x_k,x_{k+1})>0$ for $k=0,\ldots, K-1$. We can add the corresponding $L(x_k,x_{k+1})$ along these edges to construct a flow $M[x,y]$ with $\div_\SetNodes M[x,y]=\rho^y_A -\rho^x_A$.
All entries of all $M[x,y]$ are bounded (in absolute value) by $\widetilde C(\SetNodes):=\card\SetNodes/C_*$, where $C_*$ is defined in \eqref{eq:constantsC}. For $x=y$, $M[x,x]$ is simply zero.

Now assume $\massNodes_A=\rho_A^x$ for some $x \in \SetNodes$. Let
$\momentum_0 = \sum_{y \in \SetNodes} M[x,y]\,\pi(y)\,$. One finds
\begin{align*}
\div_\SetNodes \momentum_0 = \sum_{y \in \SetNodes} \left(\tfrac1{\pi(y)}  \delta_y -\tfrac1{\pi(x)}  \delta_x\right) \pi(y) = \odens-\rho_A^x\,.
\end{align*}
Again, every entry of $\momentum_0$ is bounded in absolute value by $\widetilde C(\SetNodes)$. Now let
$\momentum(t) =  2\,\momentum_0\,t$,
$\massNodes(t) =  \massNodes_A^x + (\div_\SetNodes \momentum_0)\,t^2 = (1-t^2) \cdot \rho_A^x  + t^2 \cdot \odens\,$.
We find $(\massNodes,\momentum) \in \CE(\rho_A^x,\odens)$.
One has $|\momentum(t,x,y)| \leq t \cdot 2\widetilde C(\SetNodes)$ and $\massNodes(t,x)\geq t^2$ and using the monotonicity of $\auxilaryFctAction$ for the action $\Action$ we get
\begin{align*}
\Action(\massNodes,\momentum) & \leq \frac12 \int_0^1 \sum_{x,y \in \SetNodes} \frac{(t \cdot 2 \widetilde C(\SetNodes))^2}{t^2} Q(x,y)\,\pi(x)\, \d t = 2\widetilde C(\SetNodes)^2C^*\,.
\end{align*}
\medskip

\emph{Construction of discrete counterparts:}
For fixed $h=1/N$ let $\massNodes_h=\lindiscr_h \massNodes$ and $\momentum_h = \avg_h \momentum$.
By construction $(\massNodes_h,\momentum_h) \in \CE_h(\rho_A^x,\odens)$. Then, one finds
$\momentum_h(t_i,x,y)  \leq (i+\tfrac12)\,h\,2\widetilde C(\SetNodes)$,
$\massNodes_h(t_i,x)  \geq i^2\,h^2$, 
$(\avg_h \massNodes_h)(t_i,x)  \geq (i^2 + i + \tfrac12)\,h^2\,$,
and thus
\begin{align*}
\Action_h(\massNodes_h, \momentum_h) & = \Action(\avg_h \massNodes_h,\momentum_h) \leq \frac12 \sum_{i=0}^{N-1}
h \frac{h^2\,4\widetilde C(\SetNodes)^2\,(i+\tfrac12)^2}{h^2 (i^2 + i + \tfrac12)}
\sum_{x,y \in \SetNodes} Q(x,y)\,\pi(x)
\leq 2\widetilde C(\SetNodes)^2C^*\,.
\end{align*}
\medskip

\emph{Extension for arbitrary initial data:}
For given $x \in \SetNodes$ let $(\massNodes^x,\momentum^x)$ be the (continuous) trajectory between $\rho_A^x$ and $\odens$ as constructed above. Any $\massNodes_A$ is a superposition of various $\rho^x_A$:
\begin{align*}
\massNodes_A = \sum_{x \in \SetNodes} \massNodes_A(x) \,\delta_x = \sum_{x \in \SetNodes} \massNodes_A(x) \pi(x) \,\rho^x_A
 \end{align*}
By linearity of the continuity equation the trajectory 
$(\massNodes,\momentum) = \sum_{x \in \SetNodes} \massNodes_A(x)\,\pi(x) \cdot (\massNodes^x,\momentum^x)
$
then lies in $\CE(\massNodes_A,\odens)$.
Since $\Action$ is convex and 1-homogeneous, it is sub-additive. Therefore,
\begin{align*}
\Action(\massNodes,\momentum) & \leq \sum_{x \in \SetNodes} \massNodes_A(x)\,\pi(x)\,\Action(\massNodes^x,\momentum^x) \leq \sum_{x \in \SetNodes} \massNodes_A(x)\,\pi(x)\,2\,\widetilde C(\SetNodes)^2C^*=2\,\widetilde C(\SetNodes)^2C^*\,.
\end{align*}
For the discrete trajectory the reasoning is completely
analogous. Thus the claim follows with
$C(\SetNodes)=2\,\widetilde C(\SetNodes)^2C^*$.
\end{proof}

\begin{Corollary}
	\label{cor:UniformWcBound}
	The above strategy can be used to construct trajectories between arbitrary $\massNodes_A$, $\massNodes_B$ via $\odens$ as intermediate state. This establishes that $\Wc$ and $\Wc_h$ are uniformly bounded on $\Pc(\SetNodes)^2$.
\end{Corollary}
\begin{Remark}
	In \cite{Ma11} it is shown that $\Wc$ is bounded if the constant
	$
		C_\theta := \int_0^1 \frac{1}{\sqrt{\theta(1-r,1+r)}} \d r
	$
	is finite. 
	Here, we assumed that $\theta(s,s)=s$ for $s \in \R^+_0$ and that  $s \mapsto \theta(s,t)$  is increasing on $\R^+_0$ for fixed $t \in \R^+_0$ 
	which implies that $\theta(s,t) \geq \min\{s,t\}$ for $s,\,t \in \R^+_0$. This is sufficient for $C_\theta<\infty$.
\end{Remark}
\begin{Theorem}[$\Gamma$-convergence of time discrete energies]
	\label{thm:GammaConvergence}
	Let $\massNodes_A$, $\massNodes_B$ be fixed temporal boundary conditions. Then, the sequence of functionals $(\EnergyFunctional_h)_h$ $\Gamma$-converges for $h \to 0$ to the functional $\EnergyFunctional$ with respect to the weak$\ast$ topology in $\Mc([0,1],\R^{\SetNodes} \times \R^{\SetNodes \times \SetNodes})$.
\end{Theorem}
\begin{proof}
To establish $\Gamma$-convergence, we have to verify the $\Gamma$-$\liminf$ and $\Gamma$-$\limsup$ properties.
\medskip

For the $\Gamma$-$\liminf$ property, we have to demonstrate that  the inequality
\begin{align} \label{eq:liminf}
\Action(\massNodes,\momentum)
+ \indicatorFct_{\CE(\massNodes_A,\massNodes_B)}(\massNodes,\momentum)
\leq \liminf_{h \to 0}
\Action_h(\massNodes_h,\momentum_h)
+ \indicatorFct_{\CE_h(\massNodes_A,\massNodes_B)}(\massNodes_h,\momentum_h)
\end{align}
holds for all sequences $(\massNodes_h,\momentum_h) \weaklyStar (\massNodes,\momentum)$ in $\Mc([0,1],\R^{\SetNodes} \times \R^{\SetNodes \times \SetNodes})$.
As $\CE(\massNodes_A,\massNodes_B)$ is weak-$\ast$ closed and $\CE_h(\massNodes_A,\massNodes_B) \subset \CE(\massNodes_A,\massNodes_B)$ the statement is trivial if there is no subsequence with $(\massNodes_h,\momentum_h) \in \CE_h(\massNodes_A,\massNodes_B)$. 
Thus, we may assume that all $(\massNodes_h,\momentum_h)$ fulfill the discrete continuity equation,  that $(\massNodes,\momentum)$
fulfills the continuous continuity equation, and  all $\rho_h$ are non-negative.

Now, $\massNodes_h \weaklyStar \massNodes$ implies $\avg_h \massNodes_h \weaklyStar \massNodes$ and
$\Action_h(\massNodes_h,\momentum_h) = \Action(\avg_h \massNodes_h,\momentum_h)\,$.
Since $\auxilaryFctAction$ is jointly convex and lower semi-continuous in $\massNodes$ and $\momentum$, the action functional $\Action$ is weak-$\ast$ lower semi-continuous and \eqref{eq:liminf} holds. \medskip

To verify the $\Gamma$-$\limsup$ property we need to show that for any $(\massNodes,\momentum) \in \Mc([0,1],\R^{\SetNodes} \times \R^{\SetNodes \times \SetNodes})$ there exists a recovery sequence $(\massNodes_h,\momentum_h) \weaklyStar (\massNodes,\momentum)$ with
\begin{align}\label{eq:limsup}
\limsup_{h \to 0}
\Action_h(\massNodes_h,\momentum_h)
+ \indicatorFct_{\CE_h(\massNodes_A,\massNodes_B)}(\massNodes_h,\momentum_h)
\leq \Action(\massNodes,\momentum)
+ \indicatorFct_{\CE(\massNodes_A,\massNodes_B)}(\massNodes,\momentum)\,.
\end{align}
We may assume $\Action(\massNodes,\momentum)<\infty$ and $(\massNodes,\momentum) \in \CE(\massNodes_A,\massNodes_B)$. 
Using Lemma \ref{Lemma:L2} this implies in particular that $\massNodes \in C^{0,\frac12}([0,1],\R^{\SetNodes})$.
For such a trajectory $(\massNodes,\momentum)$ we will construct a recovery sequence in two steps:
First, the continuous trajectory $(\massNodes,\momentum)$ is regularized, then, the regularized still time continuous trajectory is discretized using local averaging in time.
The regularization is necessary to control the effect of the discontinuity of $\auxilaryFctAction$ at the origin, see \eqref{def:localAction}.

\medskip

Let $(\massNodes_{A,\odens},\momentum_{A,\odens}) \in \CE(\massNodes_A,\odens)$ be the trajectory from $\massNodes_A$ to $\odens$ as constructed in Proposition \ref{prop:FiniteCostToFlatMeasure}, analogously let $(\massNodes_{\odens,B},\momentum_{\odens,B}) \in \CE(\odens,\massNodes_B)$ be the corresponding trajectory from $\odens$ to $\massNodes_B$ with $(\massNodes_{\odens,B},\momentum_{\odens,B})(t,\cdot)  := (\massNodes_{B,\odens},-\momentum_{B,\odens})(1-t,\cdot)$. Then, 
for $\delta \in (0,\frac12)$ and $\epsilon=\delta^2$ we define
\begin{align*}
\massNodes_\delta(t) & = \begin{cases}
(1-\epsilon) \cdot \massNodes_A + \epsilon \cdot \massNodes_{A,\odens}(t/\delta) &
\tn{for } t \in [0,\delta), \\
(1-\epsilon) \cdot \massNodes((t-\delta)/(1-2\delta)) + \epsilon \cdot \odens & 
\tn{for } t \in [\delta,1-\,\delta), \\
(1-\epsilon) \cdot \massNodes_B + \epsilon \cdot \massNodes_{\odens,B}((t-(1-\delta))/\delta) & \tn{for } t \in [1-\delta,1]
\end{cases} \\
\intertext{and}
\momentum_\delta(t) & = \begin{cases}
\epsilon \cdot \delta^{-1} \cdot \momentum_{A,\odens}(t/\delta) & 
\tn{for } t \in [0,\delta), \\
(1-\epsilon) \cdot (1-2\delta)^{-1} \cdot \momentum((t-\delta)/(1-2\delta)) & 
\tn{for } t \in [\delta,1-\,\delta), \\
\epsilon \cdot \delta^{-1} \cdot \momentum_{\odens,B}((t-(1-\delta))/\delta) & \tn{for } t \in [1-\delta,1]\,.
\end{cases}
\end{align*}
One finds that $(\massNodes_\delta,\momentum_\delta) \in \CE(\massNodes_A,\massNodes_B)$.
To evaluate the action of $(\massNodes_\delta,\momentum_\delta)$ we decompose it into the contributions of the time intervals $\interval{l}=[0,\delta]$, $\interval{m}=[\delta,1-\delta]$ and $\interval{r}=[1-\delta,1]$:
\newcommand{\ActionInt}{\Ac^{\tn{int}}}
\begin{align*}
\Action(\massNodes_\delta,\momentum_\delta)=\Action_l + \Action_m + \Action_r
\qquad \tn{with} \qquad
\Action_\chi = \int_{\interval{\chi}} \ActionInt(\massNodes_\delta(t),\momentum_\delta(t))\d t
\tn{ for } \chi \in \{l,m,r\}\,.
\end{align*}
where
\begin{align*}
\ActionInt & : \R^{\SetNodes} \times \R^{\SetNodes \times \SetNodes} \to \RCupInf, &
(\massNodes,\momentum) & \mapsto \frac12 \sum_{x,y \in \SetNodes} \auxilaryFctAction\big( \massNodes(x), \massNodes(y), \momentum(x,y) \big) \, Q(x,y) \, \pi(x)\,.
\end{align*}
$\ActionInt$ is jointly convex and 1-homogeneous and therefore sub-additive. Moreover, it is 2-homogeneous in the second argument. Therefore we obtain
\begin{align*}
\nonumber
\Action_m & \leq \frac{(1-\epsilon)}{(1-2\delta)^{2}} \int_{\interval{m}}
\ActionInt\big(
\massNodes((t-\delta)/(1-2\delta)),
\momentum((t-\delta)/(1-2\delta))
\big)
\d t \\\label{eq:LimSupEMBound}
&= \frac{(1-\epsilon)}{(1-2\delta)}
\int_0^1 \ActionInt\big(\massNodes(t),\momentum(t)\big) \d t
= \frac{(1-\epsilon)}{(1-2\delta)}
\Action(\massNodes,\momentum)\,.
\end{align*}
Further, using Proposition  \ref{prop:FiniteCostToFlatMeasure} we obtain $\Action_{l} + \Action_{r} \leq 2 \,C(\SetNodes) \, \delta$.

\medskip
Next, we discretize in time. Since
$\massNodes \in C^{0,\frac12}([0,1],\R^{\SetNodes})$ we have
$|\massNodes(t,x)-\massNodes(t',x)|\leq g(|t-t'|)$ with
$g(s):= C \cdot s^{\frac12}$ for all $x \in \SetNodes$. Now let
$\Delta=g(2h)$ and choose the regularization parameter
$\delta=\min\{i \cdot h \,\colon\, i \in \N,\, i\cdot h \geq
\Delta^{\frac14} \}$
and as before $\epsilon=\delta^2$. Obviously $\Delta$, $\delta$ and
$\epsilon \to 0$ as $h \to 0$. In particular, for $h$ sufficiently
small $2 \geq 1/(1-2\delta)$ and thus
$\Delta=g(2h)\geq g(h/(1-2\delta))$. Therefore, $\Delta$ is a uniform
upper bound for the variation of $\massNodes_\delta$ on any interval
of the size $h$. We now set
\begin{align*}
\massNodes_h  = \lindiscr_h \massNodes_\delta, \quad
\momentum_h  = \avg_h \momentum_\delta\,,
\end{align*}
and note that
$(\massNodes_h,\momentum_h) \in \CE_h(\massNodes_A,\massNodes_B)$.  As
$\delta \to 0$ one finds
$(\massNodes_\delta,\momentum_\delta) \weaklyStar
(\massNodes,\momentum)$
and for $h \to 0$ we obtain
$(\massNodes_\delta-\massNodes_h,\momentum_\delta-\momentum_h)
\weaklyStar 0$.
This implies that
$(\massNodes_h,\momentum_h) \weaklyStar (\massNodes,\momentum)$.
\medskip

Note that $\delta$ was chosen to be an integer multiple of $h$. So the division of $[0,1]$ into the three intervals $[0,\delta]$, $[\delta,1-\delta]$ and $[1-\delta,1]$ in the construction of $(\massNodes_\delta,\momentum_\delta)$ is compatible with the grid discretization of step size $h$.
Therefore, as above, the discrete action decomposes into three contributions which we denote 
$\Action_h(\massNodes_h,\momentum_h)=\Action_{l,h} + \Action_{m,h} + \Action_{r,h}$.
Again, using joint 1-homogeneity and sub-additivity of $\auxilaryFctAction$, as well as the 2-homogeneity in the second argument one obtains
\begin{align*}
\Action_{l,h} \leq \frac{\epsilon}{\delta} \cdot \Action_h(\lindiscr_h \massNodes_{A,\odens}, \avg_h \momentum_{A,\odens}),\quad 
\Action_{r,h} \leq \frac{\epsilon}{\delta} \cdot \Action_h(\lindiscr_h \massNodes_{\odens, B}, \avg_h \momentum_{\odens, B})\,.
\end{align*}
Using Proposition \ref{prop:FiniteCostToFlatMeasure} we observe that
\begin{align*}
\Action_{l,h} + \Action_{r,h} \leq 2 \,C(\SetNodes) \, \delta\,.
\end{align*}
In view of \eqref{eq:LimSupEMBound}, it remains to estimate
$\Action_{m,h}$ by a suitable constant times $\Action_m$.
\newcommand{\intervalsM}{S_m} To this end, let
$\intervalsM \subset \{0,\ldots,N-1\}$ the the set of discrete indices
such that $\interval{i} \subset \interval{m}$ for $i \in \intervalsM$.
Then $\Action_m$ is given by
\begin{align*}
\Action_m = \frac12 \sum_{i\in\intervalsM} \sum_{x,y \in \SetNodes}
\left[ \int_{\interval{i}}
\auxilaryFctAction\big(\massNodes_\delta(t,x),\massNodes_\delta(t,y),\momentum_\delta(t,x,y)\big)\,\d t
\right] Q(x,y)\,\pi(x)\,.
\end{align*}
Since $\auxilaryFctAction$ is convex, by Jensen's inequality one finds
\begin{align*}
\int_{\interval{i}} \auxilaryFctAction\big(\massNodes_\delta(t,x),\massNodes_\delta(t,y),\momentum(t,x,y)\big)\,\d t
&\geq h \cdot
\auxilaryFctAction\big( (\avg_h \massNodes_\delta)(t_i,x), (\avg_h \massNodes_\delta)(t_i,y), (\avg_h \momentum_\delta)(t_i,x,y) \big) \,.
\end{align*}
The discretized action $\Action_{m,h}$ is a weighted sum of the form
\begin{align*}
\Action_{m,h} & =
\frac12 \sum_{i \in \intervalsM} \sum_{x,y \in \SetNodes} 
h \cdot \auxilaryFctAction \big((\avg_h \lindiscr_h \massNodes_\delta)(t_i,x), (\avg_h \lindiscr_h \massNodes_\delta)(i,y), \avg_h m_\delta)(t_i,x,y) Q(x,y)\,\pi(x) \,.
\end{align*}
By construction $\massNodes_\delta$ is bounded from below by $\epsilon$ on $\interval{m}$ on all nodes and its variation within each discretization interval is bounded by $\Delta$. Therefore, for any $i \in \intervalsM$, $z \in \SetNodes$ one finds
\begin{align*}
(\avg_h \massNodes_\delta)(t_i,z) & \leq (\avg_h \lindiscr_h \massNodes_\delta)(t_i,z) + \Delta\,, \quad 
(\avg_h \lindiscr_h \massNodes_\delta)(t_i,z)  \geq \epsilon\,.
\end{align*}
Due to the monotonicity of $s\to \tfrac{s}{s+\Delta}$ we obtain 
\begin{align*}
\frac{(\avg_h \lindiscr_h \massNodes_\delta)(t_i,z)}{(\avg_h \massNodes_\delta)(t_i,z)}
\geq \frac{(\avg_h \lindiscr_h \massNodes_\delta)(t_i,z)}{(\avg_h \lindiscr_h \massNodes_\delta)(t_i,z) + \Delta}
\geq \frac{\epsilon}{\epsilon + \Delta}.
\end{align*}
Taking into account the joint 1-homogeneity of $\theta$ and the monotonicity of $\theta$ in each single argument this implies for all $x,y \in \SetNodes$ that
\begin{align*}
\frac{\theta\big(
(\avg_h \lindiscr_h \massNodes_\delta)(t_i,x),
(\avg_h \lindiscr_h \massNodes_\delta)(t_i,y)\big)
}{
\theta\big(
(\avg_h \massNodes_\delta)(t_i,x),
(\avg_h \massNodes_\delta)(t_i,y)\big)
}
\geq \frac{\epsilon}{\epsilon + \Delta} = \frac{1}{1 + \Delta/\epsilon}\,.
\end{align*}
Hence, 
\begin{align*}
\Action_{m,h} & =
\frac12 \sum_{i \in \intervalsM} \sum_{x,y \in \SetNodes} 
h \cdot \frac{(\avg_h m_\delta)^2(t_i,x,y)}{\theta((\avg_h \lindiscr_h \massNodes_\delta)(t_i,x), (\avg_h \lindiscr_h \massNodes_\delta)(t_i,y))} Q(x,y)\,\pi(x) \\
&\leq \frac12 (1 + \Delta / \epsilon) \sum_{i \in \intervalsM} \sum_{x,y \in \SetNodes} 
h \cdot \frac{(\avg_h m_\delta)^2(t_i,x,y)}{\theta((\avg_h  \massNodes_\delta)(t_i,x), (\avg_h  \massNodes_\delta)(t_i,y))} Q(x,y)\,\pi(x)  = (1 + \Delta / \epsilon) \Action_{m}\,.
\end{align*}
Our choice of $\delta$ implies that $\epsilon = \delta^2 \geq \Delta^\frac12$ and thus $\Delta/\epsilon \leq \epsilon$. 
Altogether, we obtain for $h$ sufficiently small 
\begin{align*}
\Action_h(\massNodes_h,\momentum_h)  =\Action_{l,h} + \Action_{m,h} + \Action_{r,h} \leq 2\, C(\SetNodes) \, \delta + 
(1 +  \epsilon) \, \frac{1-\epsilon}{1-2\delta} \, \Action(\massNodes,\momentum)\,.
\end{align*}
Since $\delta \to 0$, $\epsilon \to 0$ as $h \to 0$, this establishes the $\Gamma$-$\limsup$ property.
\end{proof}

Next, we establish convergence of the discrete optimizers to a
continuous solution.  To establish compactness we first show a uniform
bound for the $L^2$ norm of the discrete momenta, in analogy to Lemma
\ref{Lemma:L2}.

\begin{Lemma}[$L^2$ bound for the discrete momentum]
\label{Lemma:L2Disc}
Let $(\massNodes_h,\momentum_h) \in \FESpaceOneNodes \times \FESpaceZeroEdges$ with discrete energy $\EnergyFunctional_h(\massNodes_h,\momentum_h) \leq \ol{E} < \infty$.
Then, there exists a constant $\ol{M}<\infty$ only depending on $(\SetNodes,Q,\pi)$ and $\ol{E}$ (and not on $h$), 
such that $\|m_h\|_{L^2([0,1],\R^{\SetNodes \times \SetNodes})} \leq \ol{M}$.
\end{Lemma}
\begin{proof}
The proof works in complete analogy to Lemma \ref{Lemma:L2}. We bound
\begin{multline*}
\Big( \sum_{x,y \in \SetNodes} |\momentum_h(t_i,x,y)| Q(x,y) \pi(x) \Big)^2
\leq  \Big( \sum_{x,y \in \SetNodes} \auxilaryFctAction( \avg_h \massNodes_h(t_i,x),\avg_h
\massNodes_h(t_i,y), \momentum(t_i,x,y) ) Q(x,y) \pi(x) \Big) \\
 	{} \cdot  \Big( \sum_{x,y \in \SetNodes} \theta(\avg_h \massNodes_h(t_i,x), \avg_h
 		\massNodes(t_i,y)) Q(x,y) \pi(x) \Big)
\end{multline*}
and
\begin{align*}
\sum_{x,y \in \SetNodes} \theta(\avg_h \massNodes_h(t_i,x), \avg_h
 		\massNodes(t_i,y))\, Q(x,y) \pi(x)
\leq  C^*\,,
\end{align*}
where $C^*$ is defined in \eqref{eq:constantsC}. Here, we have used
that $(\massNodes_h,\momentum_h) \in \CE(\massNodes_A, \massNodes_B)$
which implies that mass is preserved, \ie
$\sum_{x \in \SetNodes} \avg_h \massNodes_h(t_i,x) \pi(x) = \sum_{x
  \in \SetNodes} \massNodes_h(t_i+\frac{h}{2},x) \pi(x) = \sum_{x \in
  \SetNodes} \massNodes_A(x) \pi(x) = 1$
for all $i=0,\ldots, N-1$, and that since
$\Action_h(\massNodes_h,\momentum_h)<\infty$ one has
$\avg_h \massNodes_h \geq 0$.  Now, once more using that $\SetNodes$
is finite and integrating (or summing) in time establishes the bound.
\end{proof}

\begin{Theorem}[Convergence of discrete geodesics]
\label{thm:Convergence} 
For fixed temporal boundary conditions $\massNodes_A$, $\massNodes_B$
any sequence $(\massNodes_h,\momentum_h)$ of minimizers of
$\EnergyFunctional_h$ is uniformly bounded in
$C^{0,\frac12}([0,1],\R^\SetNodes) \times L^2((0,1),\R^{\SetNodes
  \times \SetNodes})$
for $h \to 0$. Up to selection of a subsequence,
$\massNodes_h \to \massNodes$ strongly in
$C^{0,\alpha}([0,1],\R^\SetNodes)$ for any $\alpha\in0,\frac12)$ and
$\momentum_h \to \momentum$ weakly in $L^2$ with
$( \massNodes,\momentum)$ being a minimizer of the energy
$\EnergyFunctional$.
\end{Theorem}
\begin{proof} 
For a sequence of minimizers $(\massNodes_h,\momentum_h)$ the discrete energy $\EnergyFunctional_h(\massNodes_h,\momentum_h)$ is uniformly bounded by Corollary \ref{cor:UniformWcBound} .
Since $(\massNodes_h,\momentum_h) \in \CE(\massNodes_A,\massNodes_B)$ the total variation of all $\massNodes_h$ is uniformly bounded. Further, by Lemma \ref{Lemma:L2Disc} the $L^2$ norm $\|\momentum_h\|_{L^2([0,1],\R^{\SetNodes \times \SetNodes})}$ is uniformly bounded. 
Hence, the sequence $(\massNodes_h,\momentum_h)_h$ has a weakly$\ast$ (in the sense of measures) convergent subsequence, which by Theorem \ref{thm:GammaConvergence} and a standard  consequence of  $\Gamma$ convergence theory converges weakly$\ast$ to some minimizer $(\massNodes,\momentum)$ of $\EnergyFunctional$).
\medskip

Using the continuity equation this convergence can be strengthened.
We already know that $(\massNodes_h,\momentum_h)$ solves the continuity equation $\partial_t \massNodes_h = - \div_\SetNodes \momentum_h$. 
Thus, the uniform bound for $\momentum_h$ in $L^2((0,1),\R^{\SetNodes \times \SetNodes})$ implies that $\massNodes_h$ is uniformly bounded in $H^{1,2}(\R^\SetNodes)$.
From this we obtain by the Sobolev embedding theorem that  $(\massNodes_h)_h$ is uniformly bounded in $C^{0,\frac12}(\R^\SetNodes)$ and compact in $C^{0,\alpha}(\R^\SetNodes)$ for all $\alpha \in [0,\frac12)$.
\end{proof}

\section{Optimization with Proximal Splitting}
\label{sec:ProximalSplitting}
\subsection{Slack Variables and Proximal Splitting}
The computation of the discrete transportation distance
\eqref{def:metricDiscrete} and the associated transport path require
the solution of a finite-dimensional non-smooth convex optimization
problem.  To this end, we apply a proximal splitting approach with
suitably choosen slack variables.  The proximal mapping of a convex
and lower semi-continuous function $f:H \to \R \cup \{ \infty \}$ on a
Hilbert space $H$ with norm $\NormH{\cdot}$ is defined as
 \begin{align}
  \prox_f(x) = \argmin_{y \in H} \frac{1}{2} \NormH{x-y}^2 + f(y)\,.
 \end{align}
Furthermore, the indicator function of a closed convex set $K \subset H$ is given by 
$\indicatorFct_K(x) = 0$ for $x\in K$ and $\infty$ elsewise.
In particular, $\prox_{\indicatorFct_K} =  \proj_K$, where $\proj_K$ is the projection onto $K$. For a function $f : H \mapsto \RCupInf$ its Fenchel conjugate is given by
\begin{align}
	f^\ast(y) = \sup_{x \in H} \, \InProdH{y}{x} - f(x)\,.
\end{align}
If $f(x)<\infty$ for some $x \in H$, then $f^\ast$ is convex and lower semi-continuous.
For more details and an introduction to convex analysis see e.g.~\cite{ConvexFunctionalAnalysis-11}.
The practical applicability of proximal splitting schemes depends on whether the objective can be split into terms such that the proximal mapping for each term can be computed efficiently.
In \cite{PaPe14} a spatiotemporal discretization with staggered grids of the classical Benamou--Brenier formulation \cite{BeBr00} of optimal transport of Lebesgue densities on $\R^n$ was presented and several proximal splitting methods were considered to solve the discrete problem.
However, this approach can not directly be transfered to problem \eqref{def:metricDiscrete} since the action $\Action$ couples the variables $\massNodes$ and $\momentum$ in a non-linear way via the terms $\alpha(\momentum(t_i,x,y), \avg_h \massNodes(t_i,x), \avg_h \massNodes(t_i,y))$ spatially over the whole graph according to the transition kernel $Q$ 
and temporally via the averaging operator $\avg_h$.
Thus, the proximal mapping of the $\Action$-term is not separable in space or time and thus requires the solution of a fully coupled, nonlinear minimization problem.
As a remedy, we propose to introduce auxiliary variables to decouple the variables and rewrite the action $\Action$ with terms where variables only interact locally, thus leading to separable, hence simpler, proximal mappings.

\begin{Lemma}
For $(\massNodes,\momentum) \in \FESpaceOneNodes \times \FESpaceZeroEdges$ one finds
\begin{align}
\label{eq:actionReform1}
\Action_h(\massNodes,\momentum)  = \Action(\avg_h\massNodes,\momentum) =
\inf\left\{\ActionU(\massEdges,\momentum) + \indicatorFct_{\SetKPre}(\avg_h \massNodes,\massEdges)
\colon \massEdges \in \FESpaceZeroEdges \right\}
\end{align}
with the convex set
\begin{align}
 \SetKPre  := \left\{ (\massElement, \massEdges ) \in \FESpaceZeroNodes\times\FESpaceZeroEdges \, \colon \,
0 \leq \massEdges(t_i,x,y) \leq \theta(\massElement(t_i,x), \massElement(t_iy))
\, \forall i=0,\ldots,N-1, \ \forall x,y \in \SetNodes \right\}
\end{align}
and the edge-based action
\begin{align}
\label{def:ActionU}
\ActionU(\massEdges,\momentum) & := \frac12 \int_0^1 \sum_{x,y \in \SetNodes}
\transportCostFunction(\massEdges(t,x,y),\momentum(t,x,y)) \, Q(x,y)\,\pi(x) \d t\\
& \text{with} \quad
\transportCostFunction(\massEdges,\momentum) := \begin{cases}
\frac{\momentum^2}{\massEdges} & \text{if } \massEdges>0, \\
0 & \text{if } (\momentum,\massEdges) = (0,0), \nonumber \\
+\infty & \text{else.}
\end{cases}
\end{align}
\end{Lemma}
Note that $\transportCostFunction$ is the integrand of the Benamou--Brenier action functional and that $\auxilaryFctAction(s,t,\momentum) = \transportCostFunction(\theta(s,t),\momentum)$.
\begin{proof}
The first equality is merely the definition of $\Action_h$ and using the fact that $\avg_h \momentum = \momentum$ for $\momentum \in \FESpaceZeroEdges$.
For the second equality note that for any $\massEdges \in \FESpaceZeroEdges$ with $(\massElement,\massEdges) \in \SetKPre$ one has $\massEdges(t_i,x,y) \leq \theta(\massElement(t_i,x),\massElement(t_i,y))$. By monotonicity of $\transportCostFunction$ in its first argument this implies 
$\transportCostFunction(\massEdges(t_i,x,y),\momentum(t_i,x,y)) \geq \alpha(\massElement(t_i,x),\massElement(t_i,y),\momentum(t_i,x,y))$ and hence
\begin{align}\label{eq:Aest}
\Action(\massElement,\momentum) \leq \inf \left\{\ActionU(\massEdges,\momentum) + \indicatorFct_{\SetKPre}(\massElement,\massEdges) : \massEdges \in \R^{\SetNodes \times \SetNodes} \right\}.
\end{align}
Further, we obviously have that
$\bar\massEdges(t_i,x,y):=\theta(\massElement(t_i,x),\massElement(t_i,y))$
satisfies $(\massElement,\massEdges)\in\SetKPre$ and
$\ActionU(\bar\massEdges,\momentum)=\Action(\massElement,\momentum)$. Hence,
we have equality in \eqref{eq:Aest}.
\end{proof}

The proximal mapping of the function $\ActionU$ can be computed separately for each time interval and graph edge. 
However, the set $\SetKPre$ still couples the variables $\avg_h \massNodes$ and $\massEdges$ according to the graph structure and the averaging operator $\avg_h$ couples the variables of $\massNodes$ in time.
To resolve this, we introduce a second set of auxiliary variables.
\begin{Lemma}
For $\massElement \in \R^\SetNodes$, $\massEdges \in \R^{\SetNodes \times \SetNodes}$ one finds
\begin{align} \nonumber
\indicatorFct_{\SetKPre}(\avg_h \massNodes,\massEdges)  = \inf \Big\{ &
\indicatorFct_{\SetJavg}(\massNodes,\massElement)
+ \indicatorFct_{\SetJequal}(\massElement,\massElementSlack)
+ \indicatorFct_{\SetJRho}(\massElementSlack,\massVX,\massVY)
+ \indicatorFct_{\SetK}(\massVX,\massVY,\massEdges)
\, : \,\\
& (\massElement,\massElementSlack,\massVX, \massVY) \in
(\FESpaceZeroNodes)^2 \times (\FESpaceZeroEdges)^2
\Big\} \label{eq:actionReform2}
\end{align}
\begin{align} \intertext{where}
\SetJavg & := \left\{
(\massNodes, \massElement) \in \FESpaceOneNodes \times \FESpaceZeroNodes
\, \colon \,
\massElement = \avg_h \massNodes
     \right	\}, \\
\SetJequal & := \left\{(\massElement,\massElementSlack) \in (\FESpaceZeroNodes)^2
\,\colon\, \massElement=\massElementSlack \right\}, \\
\label{def:SetJ}
\SetJRho & := \left\{ (\massElementSlack, \massVX, \massVY)
\in \FESpaceZeroNodes \times (\FESpaceZeroEdges)^2
\,:\,
\massElementSlack(t_i,x) = \massVX(t_i,x,y), \,
\massElementSlack(t_i,y) = \massVY(t_i,x,y)
\right\},\\
\label{def:SetK}
\SetK & := \left\{ (\massVX, \massVY, \massEdges) \in (\FESpaceZeroEdges)^3
\,:\,
(\massVX(t_i,x,y), \massVY(t_i,x,y), \massEdges(t_i,x,y)) \in \SetKSmall \right\}, \\
\intertext{with}
\SetKSmall & := \{ (\massVX, \massVY, \massEdges) \in \R^3 \,:\,
0 \leq \massEdges \leq \theta(\massVX,\massVY) \}.
\end{align}
\end{Lemma}
\begin{proof}
For fixed $\massNodes \in \FESpaceOneNodes$ there is precisely one tuple $(\massElement,\massElementSlack,\massVX,\massVY)$ such that 
\[(\massNodes,\massElement) \in \SetJavg, \quad (\massElement,\massElementSlack) \in \SetJequal,\; \text{and} \quad (\massElementSlack,\massVX,\massVY) \in \SetJRho\,,\] 
given by $\massElement=\avg_h \massNodes$, $\massElementSlack=\massElement$, $\massVX(t_i,x,y)=\massElementSlack(t_i,x)$, $\massVY(t_i,x,y) = \massElementSlack(t_i,y)$. For this $(\massVX,\massVY)$ one finds $(\massVX, \massVY, \massEdges) \in \SetK$ if and only if $(\avg_h \massNodes,\massEdges) \in \SetKPre$.
\end{proof}

The function $\indicatorFct_{\SetJavg}$ relates the values of $\massNodes$ on time nodes to the average values on the adjacent time intervals, $\indicatorFct_{\SetJRho}$ communicates the values of $\massElementSlack$ on graph nodes to the adjacent graph edges and $\indicatorFct_\SetK$ ensures the mass averaging via the function $\theta$. The additional splitting via $\indicatorFct_{\SetJequal}$ will later simplify partition of the final optimization problem into primal and dual component.
The sets $\SetJavg$, $\SetJequal$, $\SetJRho$ and $\SetK$ are all products of simpler low-dimensional sets, implying simpler computation of the relevant proximal mappings and projections.

This gives us an equivalent formulation for the discrete minimization problem \eqref{def:metricDiscrete}:
\begin{align} \nonumber
\W_h(\massElement_A, \massElement_B)^2
=  \inf
\Big\{ & (\F + \G)(\massNodes_h,\momentum_h,\massEdges_h,\massVX_h,\massVY_h,\massElement_h, \massElementSlack_h)
\,:\,\\
&(\massNodes_h,\momentum_h,\massEdges_h,\massVX_h,\massVY_h,\massElement_h, \massElementSlack_h) \in \FESpaceOneNodes \times (\FESpaceZeroEdges)^4 \times (\FESpaceZeroNodes)^2
\Big\} 
\label{def:metricWithUAndV}
\end{align}
with
\begin{align*}
\F(\massNodes_h,\momentum_h,\massEdges_h,\massVX_h,\massVY_h,\massElement_h, \massElementSlack_h)
& :=
\ActionU(\massEdges_h, \momentum_h)
+ \indicatorFct_{\SetJRho}(\massElementSlack_h,\massVX_h, \massVY_h) 
+ \indicatorFct_{\SetJavg}(\massNodes_h,\massElement_h), \\
\G(\massNodes_h,\momentum_h,\massEdges_h,\massVX_h,\massVY_h,\massElement_h, \massElementSlack_h)
& :=
\indicatorFct_{\CE_h(\massNodes_A, \massNodes_B)}(\massNodes_h, \momentum_h) 
+ \indicatorFct_{\SetK}(\massVX_h,\massVY_h,\massEdges_h)
+ \indicatorFct_{\SetJequal}(\massElement_h,\massElementSlack_h).
\end{align*}

The structure of this optimization problem is well suited for the first order primal-dual algorithm presented in \cite{ChPo11}. 
We consider the Hilbert space $H = \FESpaceOneNodes \times (\FESpaceZeroEdges)^4 \times (\FESpaceZeroNodes)^2$ composed of tuples of functions in space and time
with the scalar product
\begin{align}
& \InProdH{\left( \massNodes_{h,1},\momentum_{h,1},\massEdges_{h,1},\massVX_{h,1},\massVY_{h,1},\massElement_{h,1}, \massElementSlack_{h,1}  \right)}
          {\left( \massNodes_{h,2},\momentum_{h,2},\massEdges_{h,2},\massVX_{h,2},\massVY_{h,2},\massElement_{h,2}, \massElementSlack_{h,2} \right)} \nonumber \\
& \quad := h \sum_{i=0}^N \InProdV{\massNodes_{h,1}(t_i,\cdot)}{\massNodes_{h,2}(t_i,\cdot)}
   + h \sum_{i=0}^{N-1}         
         \InProdV{\massElement_{h,1}(t_i,\cdot)}{\massElement_{h,2}(t_i,\cdot)}
         +\InProdV{\massElementSlack_{h,1}(t_i,\cdot)}{\massElementSlack_{h,2}(t_i,\cdot)} \nonumber \\
& \qquad + h \sum_{i=0}^{N-1} 
          \InProdE{\momentum_{h,1}(t_i,\cdot)}{\momentum_{h,2}(t_i,\cdot)} 
        +\InProdE{\massEdges_{h,1}(t_i,\cdot)}{\massEdges_{h,2}(t_i,\cdot)} \nonumber \\
  & \qquad + h \sum_{i=0}^{N-1}        \InProdE{\massVX_{h,1}(t_i,\cdot)}{\massVX_{h,2}(t_i,\cdot)}
         +\InProdE{\massVY_{h,1}(t_i,\cdot)}{\massVY_{h,2}(t_i,\cdot)} \,.
\label{eq:ScalarProductDiscrete}
\end{align}
and the induced norm denoted by $\NormH{\cdot}$.
Then applying \cite[Algorithm 1]{ChPo11} to solve problem \eqref{def:metricWithUAndV} with $\F, \,\G: \,H \to \RCupInf$
amounts to iteratively compute for initial data $(\iterz{a},\iterz{b}) \in H^2$ and $\iterz{\widebar{a}}=\iterz{a}$  
\begin{align}
\iterll{b} & = \prox_{\sigma\,\F^\ast}(\iterl{b}+\sigma\,\iterl{\widebar{a}}),\nonumber  \\
\iterll{a} & = \prox_{\tau\,\G}(\iterl{a}-\tau\,\iterll{b}), \label{eq:ChambollePock} \\
\iterll{\widebar{a}} & = \iterll{a} + \lambda \cdot (\iterll{a}-\iterl{a}).\nonumber
\end{align}
where $\tau$, $\sigma>0$, $\lambda \in [0,1]$ . 
As demonstrated in \cite{ChPo11}  the iterates converge to a minimizer in \eqref{def:metricWithUAndV} if $\tau \cdot \sigma < 1$.
For some $(\massNodes_h,\momentum_h,\massEdges_h,\massVX_h,\massVY_h,\massElement_h, \massElementSlack_h) \in H$ one finds
\begin{align*}
 \F^\ast(\massNodes_h,\momentum_h,\massEdges_h,\massVX_h,\massVY_h,\massElement_h, \massElementSlack_h) 
  = \ActionU^\ast(\massEdges_h,\momentum_h)
  + \indicatorFct_{\SetJRho}^\ast(\massElementSlack_h,\massVX_h, \massVY_h)
  + \indicatorFct_{\SetJavg}^\ast(\massNodes_h, \massElement_h)
\end{align*}
and the proximal mapping 
$(\massNodes^\pr_h,\momentum^\pr_h,\massEdges^\pr_h,{\massVX_h}^\pr,{\massVY_h}^\pr, \massElement^\pr_h, \massElementSlack^\pr_h) 
= \prox_{\sigma\,\F^\ast}(\massNodes_h, \momentum_h, \massEdges_h, \massVX_h, \massVY_h, \massElement_h, \massElementSlack_h)$ decomposes as follows:
\begin{align*}
 (\massEdges^\pr_h,\momentum^\pr_h)
 & = \prox_{\sigma\,\ActionU^\ast}(\massEdges_h,\momentum_h)\,, \\
 (\massElementSlack^\pr_h,{\massVX_h}^\pr,{\massVX_h}^\pr) 
 & = \prox_{\sigma\,\indicatorFct_{\SetJRho}^\ast}(\massElementSlack_h,\massVX_h,\massVY_h) \, , \\
 (\massNodes^\pr_h, \massElement^\pr_h) 
 & = \prox_{\sigma \, \indicatorFct_{\SetJavg}^\ast}(\massNodes_h,\massElement_h) \, .
\end{align*}
Likewise, for 
$(\massNodes^\pr_h,\momentum^\pr_h,\massEdges^\pr_h,{\massVX_h}^\pr,{\massVY_h}^\pr,\massElement^\pr_h, \massElementSlack^\pr_h) 
= \prox_{\tau\,\G}(\massNodes_h,\allowbreak \momentum_h,\allowbreak \massEdges_h,\allowbreak \massVX_h,\allowbreak \massVY_h, \massElement_h, \massElementSlack_h)$ one finds
\begin{align*}
 (\massNodes^\pr_h,\momentum^\pr_h) & =
 \proj_{\CE_h(\massNodes_A,\massNodes_B)}(\massNodes_h,\momentum_h)\,, \\
 ({\massVX}^\pr_h,{\massVY}^\pr_h,\massEdges_h) 
 & = \proj_{\SetK}(\massVX_h,\massVY_h,\massEdges_h)\, , \\
 (\massElement^\pr_h, \massElementSlack^\pr_h) 
 &= \proj_{\SetJequal}(\massElement_h, \massElementSlack_h) \, .
\end{align*}
Each of the proximal maps is performed with respect to the norm $\NormH{\cdot}$ restricted to the relevant variables.
\bigskip

In what follows, we will study these maps in detail.
In fact, we will observe that 
$\prox_{\ActionU^\ast}$ and $\proj_{\SetK}$ can be separated into low-dimensional problems over each time-step and edge $(x,y) \in \SetNodes \times \SetNodes$,
$\prox_{\indicatorFct_{\SetJRho}^\ast}$ splits into low-dimensional problems for each time-step and node $x \in \SetNodes$,
$\prox_{\indicatorFct_{\SetJequal}}$ is a simple pointwise update,
$\prox_{\indicatorFct_{\SetJavg}^\ast}$ decouples for each node $x \in \SetNodes$ to a sparse linear system in time,
and $\proj_{\CE_h(\massNodes_A,\massNodes_B)}$ can be computed solving a linear system, which is sparse if $Q$ is sparse.
Consequently, $\prox_{\F^\ast}$ and $\prox_\G$ can be computed efficiently and ensure that the above scheme is well-suited to solve \eqref{def:metricWithUAndV}.

\subsection[Projection onto CE]{Projection onto $\CE_h(\massNodes_A,\massNodes_B)$}
For given $(\massNodes_h,\momentum_h) \in \FESpaceOneNodes \times \FESpaceZeroEdges$ we need to solve the following problem:
\begin{align} \label{eq:ProjCE}
 \proj_{\CE_h(\massNodes_A,\massNodes_B)}(\massNodes_h,\momentum_h) 
 = \argmin_{(\massNodes^\pr_h,\momentum^\pr_h) \in  \CE_h(\massNodes_A,\massNodes_B)}
 \frac{h}{2} \sum_{i=0}^N \NormV{\massNodes^\pr_h(t_i,\cdot)-\massNodes_h(t_i,\cdot)}^2 + \frac{h}{2} \sum_{i=0}^{N-1} \NormE{\momentum^\pr_h(t_i,\cdot)-\momentum_h(t_i,\cdot)}^2
\end{align}
To this end we take into account the following dual formulation.
\begin{Proposition}\label{prop:ProjCEDual}
The solution $(\massNodes^\pr_h,\momentum^\pr_h)$ to \eqref{eq:ProjCE} is given by
\begin{subequations}
\label{eq:ProjCEPDRelationGF}
\begin{align}
 \label{eq:ProjCEPDRelationGF:A}
 \massNodes^\pr_h(t_i,x) & = \massNodes_h(t_i,x) + \frac{\LagrangeMultiplierCE_h(t_i,x) - \LagrangeMultiplierCE_h(t_{i-1},x)}{h} \,, \quad \forall i=1,\ldots,N-1 \, , \\
 \label{eq:ProjCEPDRelationGF:B}
 \massNodes^\pr_h(t_0,x) & = \massNodes_A(x) \, , \; \massNodes^\pr_h(t_N,x) = \massNodes_B(x) \, , \\
 \label{eq:ProjCEPDRelationGF:C}
 \momentum^\pr_h(t_i,x,y) &= \momentum_h(t_i,x,y) + \nabla_\SetNodes \LagrangeMultiplierCE_h(t_i,x,y)  \,, \quad \forall i=1,\ldots,N-1 \, .
\end{align}
\end{subequations}
where $\LagrangeMultiplierCE_h$  solves the space time elliptic equation

\begin{align}
& \pi(x) \frac{ \LagrangeMultiplierCE_h(t_1,x) - \LagrangeMultiplierCE_h(t_{0},x) }{h^2} + \pi(x)\,\triangle_\SetNodes \LagrangeMultiplierCE_h(t_0,x)
  = - \pi(x) \left( \frac{\massNodes_h(t_{1},x) - \massNodes_A(x)}{h} + \div \momentum_h(t_0,x) \right) \; , \nonumber \\
& \pi(x) \frac{ -\LagrangeMultiplierCE_h(t_{N-1},x) + \LagrangeMultiplierCE_h(t_{N-2},x)}{h^2} + \pi(x) \triangle_\SetNodes \LagrangeMultiplierCE_h(t_{N-1},x)  \nonumber \\
& \qquad \qquad  \qquad \qquad \qquad  \qquad \qquad 
  = - \pi(x) \left( \frac{\massNodes_B(x) - \massNodes_h(t_{N-1},x)}{h} + \div \momentum_h(t_{N-1},x) \right) \nonumber \\
 &\pi(x) \frac{\LagrangeMultiplierCE_h(t_{i+1},x) - 2 \LagrangeMultiplierCE_h(t_i,x) + \LagrangeMultiplierCE_h(t_{i-1},x) }{h^2} + 
 \pi(x)\,\triangle_\SetNodes \LagrangeMultiplierCE_h(t_i,x)   \nonumber\\
& \qquad \qquad  \qquad \qquad \qquad  \qquad \qquad = - \pi(x) \left( \frac{\massNodes_h(t_{i+1},x) - \massNodes_h(t_i,x)}{h} + \div \momentum_h(t_i,x) \right) \label{eq:ProjCELM}
\end{align}
for $i=1,\ldots,N-2$ and $x \in \SetNodes$.
\end{Proposition}
The factors $\pi(x)$ in \eqref{eq:ProjCELM} could be canceled but they will simplify further analysis.
\begin{proof}
We define the Lagrangian corresponding to \eqref{eq:ProjCE} as
\begin{align*}
 L[\massNodes_h^\pr, \momentum_h^\pr,\LagrangeMultiplierCE_h,\lambda_A,\lambda_B]
 =&  \frac{h}{2} \sum_{i=0}^N \NormV{\massNodes^\pr_h(t_i,\cdot)-\massNodes_h(t_i,\cdot)}^2 + \frac{h}{2} \sum_{i=0}^{N-1} \NormE{\momentum^\pr_h(t_i,\cdot)-\momentum_h(t_i,\cdot)}^2 \\
 & +h \sum_{i=0}^{N-1} \sum_{x \in \SetNodes} \LagrangeMultiplierCE_h(t_i,x) \left( \frac{\massNodes_h^\pr(t_{i+1},x) - \massNodes_h^\pr(t_i,x)}{h} + \div_\SetNodes \momentum_h^\pr(t_i,x) \right) \pi(x) \\
 &+  \sum_{x \in \SetNodes} \left( \lambda_B(x) ( \massNodes_h(t_N,x) - \massNodes_B(x) ) + \lambda_A(x) ( \massNodes_h(t_0,x) - \massNodes_A(x) ) \right) \pi(x)
\end{align*}
where $\lambda_A, \lambda_B$ are the Lagrange multipliers for the boundary conditions $\massNodes_h(t_0,\cdot) = \massNodes_A$, $\massNodes_h(t_N,\cdot) = \massNodes_B$.
The optimality condition in $\massNodes_h$ and $\momentum_h$ directly imply \eqref{eq:ProjCEPDRelationGF:A} and \eqref{eq:ProjCEPDRelationGF:C}. \eqref{eq:ProjCEPDRelationGF:B} reflects the boundary conditions, which are to be ensured in $\CE_h(\massNodes_A,\massNodes_B)$.
Inserting these relations into the continuity equation $\partial_t \massNodes_h^\pr + \div \momentum_h^\pr=0$ leads to the system of equations \eqref{eq:ProjCELM}.
\end{proof}

The Lagrange multiplier $\LagrangeMultiplierCE_h$ in Proposition \ref{prop:ProjCEDual} lives in $\FESpaceZeroNodes$ which can be identified with $\R^{N  \card \SetNodes}$. We equip this space with the canonical basis
\begin{align*}
(\LagrangeMultiplierCE_h^{i,x})_{i=0,\ldots, N-1,\,x\in \SetNodes}
\quad \tn{where} \quad
(\LagrangeMultiplierCE_h^{i,x})(t_j,y) = \delta_{i,j} \cdot \delta_{x,y}
\end{align*}
and the standard Euclidean inner product with respect to this basis.
Then the elliptic equation \eqref{eq:ProjCELM} can be written as a linear system $S Z = F$ 
for a coordinate vector $Z = (\LagrangeMultiplierCE_h(t_i,x))_{i=0,\ldots N-1,\, x\in X}$,
a matrix $S \in \R^{(N\card \SetNodes) \times (N \card \SetNodes)}$
and a vector $F \in \R^{N  \card \SetNodes}$.
The matrix $S$ is symmetric since $\pi(x)Q(x,y) = \pi(y) Q(y,x)$ and the matrix representation of $\triangle_\SetNodes$ is $Q - \text{diag}(\sum_yQ(\cdot,y))$.
Furthermore, $S$ is sparse if $Q$ is sparse.
However, the matrix $S$ is not invertible, its kernel is spanned by functions that are constant in space and time. 
To see this, assume that a non constant $Z$  
is in the kernel of $S$ and denote by $\phi_h$ the associated function in $\FESpaceZeroNodes$. Now, 
let  $I_+(\mu) := \{(i,x) \in \{0,\ldots, N-1\} \times \SetNodes\,:\,\phi_h(i,x)>\mu\}$ for $\mu = \min \phi_h(i,x)$
and define $\psi_h\in \FESpaceZeroNodes$ via $\psi(t_i, x) = 1$ if $(i,x)\in I_+(\mu)$ and $\psi_h(t_i, x) = 0$ else. 
Let $W$ be the associated nodal vector to $\psi$. 
By assumption on $Z$ the set $I_+(\mu)$ is non empty and thus it is easy to see that $W^\top S Z <0$ and thus $Z$ 
can not be in the kernel of $S$, which proves the claim.

We impose the additional constraint $\sum_{i=0}^{N-1} \sum_{x \in \SetNodes} \LagrangeMultiplierCE_h(t_i,x) = 0$ to remove this ambiguity. 
This can be written as $w^\top \LagrangeMultiplierCE_h=0$ where $w \in \R^{N  \card \SetNodes}$ is the vector with entries $w^{i,x} = 1$ leading to the linear system 
\begin{align*}
 \begin{pmatrix}
  S & w \\
  w^T & 0 
 \end{pmatrix}
 \begin{pmatrix}
  Z\\
  \lambda
 \end{pmatrix}
 =
 \begin{pmatrix}
  F \\
  0
 \end{pmatrix}
 \, .
\end{align*}
This system is uniquely solvable and the solution implies $\lambda=0$ if $F \perp w$ (in the Euclidean sense), which is true because $\massNodes_A$ and $\massNodes_B$ are assumed to be of equal mass.

\subsection[{Proximal Mapping of A*}]{Proximal Mapping of $\ActionU^\ast$}
\newcommand{\SetBB}{\Bc}
The function $\ActionU$ is convex and 1-homogeneous, hence its Fenchel conjugate is the indicator function of a convex set and the proximal mapping of $\ActionU^\ast$ is a projection.
For $(\massEdges,\momentum) \in (\FESpaceZeroEdges)^2$ one has
\begin{align*}
\ActionU(\massEdges,\momentum) & = h \sum_{i=0}^{N-1} \sum_{x,y \in \SetNodes}
\transportCostFunction(\massEdges(t_i,x,y),\momentum(t_i,x,y)) \, Q(x,y)\,\pi(x)\,.
\end{align*}
Following \cite{BeBr00} a direct calculation for $(p,q) \in (\FESpaceZeroEdges)^2$ yields
\begin{align*}
\ActionU^\ast(p,q) & = \sup_{(\massEdges,\momentum) \in (\FESpaceZeroEdges)^2}
h \sum_{i=0}^{N-1}
\Big[
\InProdE{p(t_i,\cdot,\cdot)}{\massEdges(t_i,\cdot,\cdot)}
+ \InProdE{q(t_i,\cdot,\cdot)}{\momentum(t_i,\cdot,\cdot)} \\[-1ex]
& \qquad \qquad \qquad \qquad - \frac12 \sum_{(x,y) \in \SetNodes \times \SetNodes} \transportCostFunction(\massEdges(t_i,x,y),\momentum(t_i,x,y))\,Q(x,y)\,\pi(x) \Big] \\
& = \frac{h}{2} \sum_{ \genfrac{}{}{0pt}{}{i=0,\ldots, N-1 }{(x,y) \in \SetNodes \times \SetNodes} }
\transportCostFunction^\ast(p(t_i,x,y),q(t_i,x,y))\,Q(x,y)\,\pi(x)
=\sum_{ \genfrac{}{}{0pt}{}{i=0,\ldots, N-1 }{(x,y) \in \SetNodes \times \SetNodes} }
\indicatorFct_\SetBB(p(t_i,x,y),q(t_i,x,y))
\end{align*}
with $\transportCostFunction^\ast=\indicatorFct_\SetBB$ for $\SetBB = \{ (p,q) \in \R^2\,:\, p+\frac{q^2}{4} \leq 0 \}\,$.
Thus the proximal mapping separates into two-dimensional problems for each time interval and graph edge and $(p^{\pr},q^{\pr}) = \prox_{\sigma\,\ActionU^\ast}(p,q)$ precisely if
\begin{align*}
(p^{\pr}(t_i,x,y),q^{\pr}(t_i,x,y)) = \proj_{\SetBB}(p(t_i,x,y),q(t_i,x,y))\,,
\end{align*}
where $\proj_{\SetBB}$ is the projection with respect to the standard Euclidean distance on $\R^2$ and  a Newton scheme in $\R$ can be used 
to solve for this projection.
Since this proximal mapping is a projection, it is in particular independent of the step size $\sigma$.

\subsection[Projection onto K]{Projection onto $\SetK$}
For given $(\massVX,\massVY,\massEdges) \in (\FESpaceZeroEdges)^3$ we need to solve
\begin{align*}
\proj_{\SetK}(\massVX,\massVY,\massEdges) =\!\! \argmin_{
({\massVX}^\pr,{\massVY}^\pr,\massEdges^\pr) \in \SetK}
\frac{h}{2} \sum_{i=0}^{N-1} \Big( &
\NormE{{\massVX}^\pr(t_i,\cdot,\cdot)-\massVX(t_i,\cdot,\cdot)}^2
+ \NormE{{\massVY}^\pr(t_i,\cdot,\cdot)-\massVY(t_i,\cdot,\cdot)}^2 \\[-1ex]
& + \NormE{\massEdges^\pr(t_i,\cdot,\cdot)-\massEdges(t_i,\cdot,\cdot)}^2
\Big)\,.
\end{align*}
Recall that $\SetK$ is a product of the tree-dimensional closed convex set $\SetKSmall$, as indicated in \eqref{def:SetK}. Therefore 
$({\massVX}^\pr,{\massVY}^\pr,\massEdges^\pr) = \proj_{\SetK}(\massVX,\massVY,\massEdges)$
decouples into the edgewise projection in each time step, i.e.
\begin{align*}
({\massVX}^\pr(t_i,x,y),{\massVY}^\pr(t_i,x,y),\massEdges^\pr(t_i,x,y)) = \proj_{\SetKSmall}(\massVX(t_i,x,y),\massVY(t_i,x,y),\massEdges(t_i,x,y))
\end{align*}
where this projection is with respect to the standard Euclidean distance on $\R^3$.
Let us denote by $\supdiff \theta(x)$ the super-differential of $\theta$ at $x \in \R^2$, which is the analogue of the sub-differential for concave functions.
More precisely, $\supdiff \theta(x) = -\partial (-\theta)(x)$, where $\partial (-\theta)(x)$ is the sub-differential of the convex function $x \mapsto -\theta(x)$ at $x$.
Then the projection $\pProj = \proj_{\SetKSmall}(p)$ of  $p \in \R^3$ is characterized by \cite[Prop.~ 6.46]{ConvexFunctionalAnalysis-11}
\begin{align}
\label{def:nCone}
p - \pProj & \in \nCone_\SetKSmall(\pProj) :=
\{ z \in \R^3 \,:\, \langle z, q-\pProj \rangle \leq 0 \, \forall \, q \in \SetKSmall \}\,,
\end{align}
where $\nCone_\SetKSmall(\pProj)$ is the normal cone of $\SetKSmall$ at $\pProj$.
To solve this inclusion we distinguish the following cases:
\begin{Lemma}
\label{lem:ProjKNCones}
For an averaging function $\theta:\R^2 \to \R$ fulfilling the assumptions listed in Section \ref{sec:intro} 
and for $\SetKSmall := \{ p \in \R^3\, : \, 0 \leq p_3 \leq \theta(p_1,p_2) \}$
the normal cone $\nCone_\SetKSmall(\pProj)$ for $\pProj \in \SetKSmall$ is given by:
\begin{enumerate}[label=(\roman*)]
\item 
\label{item:LemmaNConeTrivial}
Trivial projection: $p = \pProj \in \inter \SetKSmall \!= \! \{ (p_1,p_2,p_3) \in \R^3 : 0 \! <\!  p_3 \! < \! \theta(p_1,p_2)\}$, then $\nCone_\SetKSmall(\pProj) = \{0\}$.
\item 
\label{item:LemmaNConeBottom}
Projection onto `bottom facet' of $\SetKSmall$: $\pProj \in (0,+\infty) \times (0,+\infty) \times \{0\}$, then
$\nCone_\SetKSmall(\pProj) = \{0\} \times \{0\} \times \R^-_0$.
\item 
\label{item:LemmaNConeAxis}
Projection onto coordinate axis: $\pProj = (\pProj_1,0,0)$ for $\pProj_1 \in (0,+\infty)$, then 
\begin{align*}
\nCone_\SetKSmall(\pProj)= \{0\} \times \R^-_0 \times \R^-_0
\cup \left\{ (0,q_2,q_3) \in \{0\} \times \R^-_0 \times (0,+\infty) \,:\, 
(0,-q_2/q_3) \in \supdiff \theta(\pProj_1,0) \right\}.
\end{align*}
Note that $(0,q) \in \supdiff \theta(\pProj_1,0)$ is equivalent to $q \geq \lim_{z \searrow 0} \partial_2 \theta(\pProj_1,z)$ and that $\supdiff \theta(\pProj_1,0)$ is empty if $\lim_{z \searrow 0} \partial_2 \theta(\pProj_1,z)=\infty$.
The analogous representation holds for the second axis.
\item \label{item:LemmaNConeOrigin}
Projection onto origin: $\pProj = (0,0,0)$, then
\[
\nCone_\SetKSmall(\pProj)= (\R^-_0)^3 \cup \{ (q_1,q_2,q_3) \in  \R^-_0 \times \R^-_0 \times (0,+\infty) \,:\,
(q_1/q_3,q_2/q_3) \in -\supdiff \theta(0) \}\,.
\]
\item 
\label{item:LemmaNConeTop}
Projection onto `upper surface' of $\SetKSmall$: $\pProj = (\pProj_1,\pProj_2,\theta(\pProj_1,\pProj_2))$ for $(\pProj_1,\pProj_2) \in (0,+\infty)^2$, then
\[
\nCone_\SetKSmall(\pProj)=\{ \lambda \cdot (-\partial_1 \theta(\pProj_1,\pProj_2),-\partial_2 \theta(\pProj_1,\pProj_2),1) \,:\, \lambda \in \R^+\}\,.
\]
\end{enumerate}
\end{Lemma}
\begin{proof}
For $\pProj \in \inter \SetKSmall$ one finds $\nCone_\SetKSmall(\pProj)=\{0\}$ and thus $\pProj=p$, which implies \ref{item:LemmaNConeTrivial}.

In case \ref{item:LemmaNConeBottom} the set $\R \times \R \times \{0\}$ is obviously the only supporting plane of $\SetKSmall$ that contains $\pProj$. Thus the normal cone is just the ray
in direction $(0,0,-1)$.

Assume $\pProj=(\pProj_1,0,0)$, $\pProj_1>0$. Then there is some $\varepsilon>0$ such that $\{(\pProj_1+\varepsilon,0,0),(\pProj_1-\varepsilon,0,0),(\pProj_1,\varepsilon,0)\} \subset \SetKSmall$.
Therefore $\nCone_\SetKSmall(\pProj) \subset \{0\} \times \R^-_0 \times \R$.
Since $\R \times \{0\} \times \R$ and $\R \times \R \times \{0\}$ are supporting planes of $\SetKSmall$ that contain $\pProj$, one must have $\{0\} \times \R^-_0 \times \R^-_0 \subset \nCone_\SetKSmall(\pProj)$.
Moreover, for $\lim_{z \searrow 0} \partial_2 \theta(\pProj_1,z)<\infty$ 
let $z=(z_1,z_2) \in \supdiff \theta(\pProj_1,0)$. One must have $z_1=0$ and $z_2 \in \supdiff f(0)$ with auxiliary function $f : t \mapsto \theta(\pProj_1,t)$.
Then $\{ q \in \R^3 \,:\, \langle q-\pProj,(0,-z_2,1)\rangle = 0\}$ is a supporting plane of $\SetKSmall$ and consequently $(0,-z_2,1) \in \nCone_\SetKSmall(\pProj)$. Conversely, from $z_2 \notin \supdiff f(0)$ follows $(0,-z_2,1) \notin \nCone_\SetKSmall(\pProj)$.
So \[\nCone_\SetKSmall(\pProj) = \{0\} \times \R^-_0 \times \R^-_0 \cup \{(0,-\lambda \cdot z, \lambda)\,:\, z \in \supdiff f(0), \lambda \in (0,+\infty) \}.\]
The auxiliary function $f$ is concave and by monotonicity of the super-differential we find $\supdiff f(0) = [\lim_{z \searrow 0} \partial_2 \theta(\pProj_1,z), +\infty)$. 
With this characterization we arrive at the expression for $\nCone_\SetKSmall(\pProj)$ as given in \ref{item:LemmaNConeAxis}. 
The proof for the second axis is analogous.

For $\pProj = (0,0,0)$ we find $(\R^-_0)^3 \subset \nCone_\SetKSmall(0) \subset \R^-_0 \times \R^-_0 \times \R$ 
with arguments analogous to those in case \ref{item:LemmaNConeAxis}. 
For every $z =(z_1,z_2) \in \supdiff \theta(0)$ a supporting plane through $0$ is given by $\{q \in \R^3 : \langle q, (-z_1,-z_2,1) \rangle = 0 \}$ and hence $(-z_1,-z_2,1) \in \nCone_\SetKSmall(0)$. Conversely, $z=(z_1,z_2) \notin \supdiff \theta(0)$ implies $(-z_1,-z_2,1) \notin \nCone_\SetKSmall(0)$. With this, one obtains the expression for $\nCone_\SetKSmall(0)$ given in \ref{item:LemmaNConeOrigin}.

Finally, we consider $\pProj = (\pProj_1,\pProj_2,\theta(\pProj_1,\pProj_2))$ with $(\pProj_1,\pProj_2) \in (0,+\infty)^2$. 
In a neighbourhood of $\pProj$, $\SetKSmall$ is the subgraph of a concave, differentiable function. 
The unique supporting plane of $\SetKSmall$ through $\pProj$ is given by 
$\{ q \in \R^3 : \langle q-\pProj, (-\partial_1 \theta(\pProj_1,\pProj_2),-\partial_2 \theta(\pProj_1,\pProj_2),1) \rangle = 0\}$
and $(-\partial_1 \theta(\pProj_1,\pProj_2),-\partial_2 \theta(\pProj_1,\pProj_2),1)$ is the unique associated outer normal as stated in \ref{item:LemmaNConeTop}.
\end{proof}

Using Lemma \ref{lem:ProjKNCones} one can devise an algorithm for the projection onto $\SetKSmall$.
For $p=(p_1,p_2,p_3) \in \R^3$ the projection $\pProj=\proj_{\SetKSmall}(p)$ can be determined as follows:
\bigskip

\hspace{10ex} \begin{minipage}{0.8\linewidth}
\begin{algorithmic}
\Function{Project$\SetKSmall$}{$p_1$,$p_2$,$p_3$}
\LineIf{$0 \leq p_3 \leq \theta(p_1,p_2)$} \Return $(p_1,p_2,p_3)$
\LineIf{$p_3 \leq 0$} \Return $(\max\{p_1,0\},\max\{p_2,0\},0)$
\If{$(p_1 > 0) \wedge (p_2 \leq 0)$} 
\LineIf{$-p_2/p_3 \geq \lim_{z \searrow 0} \partial_2 \theta(p_1,z)$}
\Return $(p_1,0,0)$
\EndIf
\If{$(p_1 \leq 0) \wedge (p_2 > 0)$}
\LineIf{$-p_1/p_3 \geq \lim_{z \searrow 0} \partial_1 \theta(z,p_2)$}
\Return $(0,p_2,0)$
\EndIf 
\If{$(p_1 \leq 0) \wedge (p_2 \leq 0)$}
\LineIf{$(-p_1/p_3,-p_2/p_3) \in \supdiff \theta(0)$}
\Return $(0,0,0)$
\EndIf
\State \Return \Call{Project$\SetKSmall$Top}{$p_1$,$p_2$,$p_3$}
\EndFunction
\end{algorithmic}
\end{minipage}
\bigskip

The function \Call{Project$\SetKSmall$Top}{$p_1$,$p_2$,$p_3$}  in the above algorithm corresponds to case \ref{item:LemmaNConeTop} of Lemma \ref{lem:ProjKNCones}, where $\pProj$ lies on the `upper surface' of $\SetKSmall$, defined by the graph surface of  $\theta$. It will be described in more detail below.
In the following we will occasionally use the curve $c: (0,\infty) \to \R^2; \; q \mapsto (q^{-1/2},q^{1/2})$ to parametrize orientations in $(0,\infty)^2$. Due to the 1-homogeneity of $\theta$, often it suffices to look at its values at $\theta(c(q))$. Alternative choices for $c$ are feasible as well.

\begin{Lemma}[Projection onto `upper surface' of $\SetKSmall$]\label{lem:SetKSmall}
Let $p \in \R^3$ with projection on $\SetKSmall$ given by $\pProj = (\pProj_1,\pProj_2,\theta(\pProj_1,\pProj_2))$ with $(\pProj_1,\pProj_2) \in (0,+\infty)^2$.
Further, let $w(q) = (q^{1/2},q^{-1/2},\theta(q^{1/2},q^{-1/2}))$ be a parametrized curve on the `upper surface' 
and $n(q) = (-\partial_1 \theta(q^{1/2},q^{-1/2}), -\partial_2 \theta(q^{1/2},q^{-1/2}), 1)$ be the corresponding normal.
Then there exists a unique $(q,\tau) \in (0,\infty)^2$ s.t. $\pProj = \tau w(q)$. We have that $q$ is the unique root of $q \mapsto \langle p  , w(q)  \times  n(q)\rangle$
and $\tau = \langle p , \frac{w(q)}{\|w(q)\|^2}\rangle$.
\end{Lemma}
\begin{proof}
Since $\theta$ is 1-homogeneous, any $\pProj$ of the form $(\pProj_1,\pProj_2,\theta(\pProj_1,\pProj_2))$, $(\pProj_1,\pProj_2) \in (0,+\infty)^2$, can be written as $\pProj = \tau \cdot w(q)$ for unique $q \in (0,+\infty)$ and $\tau \in (0,+\infty)$. In explicit, $q= \pProj_1/\pProj_2$ and $\tau=(\pProj_1 \cdot \pProj_2)^\frac12$. 
Now, $n(q)$ is orthogonal on the graph of $\theta$ and outward pointing.
Hence, $p$ lies in the plane spanned by $w(q)$ and $n(q)$.
This is equivalent to $\langle p, w(q) \times n(q) \rangle = 0$.
Since $\pProj$ is unique, this must be the unique root of $q \mapsto \langle p, w(q) \times n(q) \rangle$.
Once $q$ is determined, we know the ray on which $\pProj$ lies. To find $\tau$, one must solve the remaining one-dimensional projection onto the ray. 
Consequently, $\tau$ is the unique minimizer of $\tau \mapsto \frac{1}{2} \|p - \tau \cdot w(q)\|^2$, which concludes the proof.
\end{proof}
For case \ref{item:LemmaNConeOrigin} of Lemma \ref{lem:ProjKNCones}  we need to characterize the super-differential of $\theta$ at the origin.
\begin{Lemma} \label{lem:ProjKSupDiff}
The super-differential of $\theta$ at the origin is given by
\[\supdiff \theta(0) = \ol{\{ \nabla \theta(q^{-1/2},q^{1/2}) \,:\, q \in (0,\infty) \} } +  (\R^+_0)^2\,.\]
\end{Lemma}
\begin{proof}
Due to the 1-homogeneity  of $\theta$
\begin{align*}
\langle \nabla \theta(\lambda p), \lambda r \rangle = \lim_{\epsilon \to 0} \frac{\theta(\lambda (p+\epsilon r))-\theta(\lambda p)}{\epsilon} 
= \lambda \lim_{\epsilon \to 0} \frac{\theta(p+\epsilon r)- \theta(p)}{\epsilon}  = \lambda\langle \nabla \theta( p), r \rangle
\end{align*}
for $p \in (0,+\infty)^2$, $\lambda >0$, and all $r\in \R^2$, which leads to $\nabla \theta(\lambda p) = \nabla \theta(p)$ for $p \in (0,+\infty)^2$ and $\lambda >0$.
Thus, for the curve $c: (0,\infty) \to \R^2; \; q \mapsto (q^{-1/2},q^{1/2})$ the set of tangent planes at $(c(q),\theta(c(q)))$ spanned by $(\nabla \theta(c(q)),1)$ and $(c(s),\theta(c(q)))$ for $q\in (0,\infty)$
is already the complete set of affine tangent planes to the graph of $\theta$ over $(0,\infty)^2$. Thus, by continuity of $\theta$ on $[0,\infty)^2$ 
we get $\theta(0) + \langle r,p\rangle \geq \theta(p)$ for $r \in \{ \nabla \theta(c(q)) \,:\, q \in (0,\infty) \}$. From this we deduce that
$\supdiff \theta(0) \supset  \{ \nabla \theta(c(q)) \,:\, q \in (0,\infty) \}  +  (\R^+_0)^2\,$.
Since $\supdiff \theta(0)$ is a closed set \cite[Prop.~16.3]{ConvexFunctionalAnalysis-11}, this implies
\[\supdiff \theta(0) \supset \ol{ \{ \nabla \theta(c(s)) \,:\, q \in (0,\infty) \} } + (\R^+_0)^2\,.\]
Furthermore, for any $w\in \R^2\setminus \{0\}$ with $w_1,\, w_2 \leq 0$ there exists a $p'$ with  
$\theta(0) + \langle (r+w),p\rangle < \theta(p')$
Since $\theta(z)=0$ for $z \in (\{0\} \times \R^+_0) \cup (\R^+_0 \times \{0\})$ and $\theta(z)=-\infty$ outside $[0,\infty)^2$ we finally obtain
that $\theta(0) + \langle r,p\rangle \geq \theta(p)$ if and only if $r \in \ol{ \{ \nabla \theta(c(q)) \,:\, q \in (0,\infty) \} } + (\R^+_0)^2$, which proves the claim.
\end{proof}

\paragraph*{Logarithmic Mean.}
Now, we turn to the specific case when $\theta=\theta_{\tn{log}}$ is the logarithmic mean \eqref{def:means}.
For $s>0$ $\lim_{t \searrow 0} \partial_1 \theta(t,s)= \lim_{t \searrow 0} \partial_2 \theta(s,t) = +\infty$. That is, $\nCone_\SetKSmall(s,0,0)=\{0\} \times \R^-_0 \times \R^-_0$ and analogous  
$\nCone_\SetKSmall(0,s,0)=  \R^-_0 \times \{0\} \times \R^-_0$. Consequently, the algorithm simplifies as follows:
\bigskip

\hspace{10ex} \begin{minipage}{0.8\linewidth}
\begin{algorithmic}
\Function{Project$\SetKSmall$}{$p_1$,$p_2$,$p_3$}
\LineIf{$0 \leq p_3 \leq \theta(p_1,p_2)$} \Return $(p_1,p_2,p_3)$
\LineIf{$p_3 \leq 0$} \Return $(\max\{p_1,0\},\max\{p_2,0\},0)$
\LineIf{$(p_1 \leq 0) \wedge (p_2 \leq 0) \wedge (-p_1/p_3,-p_2/p_3) \in \supdiff \theta(0)$}
\Return $(0,0,0)$
\State \Return \Call{Project$\SetKSmall$Top}{$p_1$,$p_2$,$p_3$}
\EndFunction
\end{algorithmic}
\end{minipage}
\bigskip

The inclusion in $\supdiff \theta(0)$  can be tested as follows.
\begin{Lemma}
\label{lem:ProjKLogSupDiff}
Let $z=(z_1,z_2) \in \R^2$. If $\min\{z_1,z_2\} \leq 0$ then $z \notin \supdiff \theta(0)$. Otherwise, there is a unique $q_1 \in (0,+\infty)$ such that $\partial_1 \theta(q_1^{-1/2},q_1^{1/2})=z_1$ and then
$z \in \supdiff \theta(0)$ if and only if  $z_2 \geq \partial_2 \theta(q^{-1/2},q^{1/2})$.
\end{Lemma}
\begin{proof}
Note that for the logarithmic mean
$\supdiff \theta(0) \subset (0,+\infty)^2$
and therefore $z \notin \supdiff \theta(0)$ if $\min\{z_1,z_2)\leq 0$.
One finds that \[\partial_1 \theta(q^{-1/2},q^{1/2})=\frac{q-1-\log(q)}{\log^2(q)}\] is monotone increasing with $\partial_1 \theta(q^{-1/2},q^{1/2}) \to 0$ as $q \to 0$ and $\partial_1 \theta(q^{-1/2},q^{1/2}) \to +\infty$ as $q \to +\infty$.
Indeed,  for $\beta(q)=\partial_1 \theta(q^{-1/2},q^{1/2})$ with $\beta(1):=\frac12$ we obtain a continuous extension on $(0,\infty)$. Furthermore, we consider 
$\beta'(q) = \frac{2(1-q) + \log(q)(1+q)}{q \log^3(q)}$ with continuous extension $\frac16$ for $q=1$ and verify that $2(1-q) + \log(q)(1+q)$ is negative for $q<1$ and 
positive for  and $q>1$. This implies that $\beta'(q) >0$.
Furthermore, by symmetry we obtain that $\partial_2 \theta(q^{-1/2},q^{1/2})$ is monotone decreasing with $\partial_2 \theta(q^{-1/2},q^{1/2}) \to +\infty$ as $q \to 0$ and $\partial_2 \theta(q^{-1/2},q^{1/2}) \to 0$ as $q \to +\infty$.  
By Lemma \ref{lem:ProjKSupDiff} 
\begin{align*}
\supdiff \theta(0) = \{ \nabla \theta(q^{-1/2},q^{1/2}) \,:\, q \in (0,+\infty) \} +  (\R_+)^2\,.
\end{align*}	
Thus, for every $z\in (0,+\infty)^2$ there is a unique $q_1\in (0,+\infty)$
such that $\partial_1 \theta(q_1^{-1/2},q_1^{1/2})=z_1$ and $z_1 \geq \partial_1 \theta(q^{-1/2},q^{1/2})$ if and only if $q\leq q_1$.
Furthermore, there is a unique $q_2\in (0,+\infty)$
such that $\partial_2 \theta(q_2^{-1/2},q_2^{1/2})=z_2$ and $z_2 \geq \partial_2 \theta(q^{-1/2},q^{1/2})$ if and only if $q\geq q_2$.
Hence, $z\in \supdiff \theta(0)$ if and only if  $q_2 \leq q_1$, which is equivalent to $z_2 \geq \partial_2 \theta(q_1^{-1/2},q_1^{1/2})$.
\end{proof}
\begin{Remark}[Comments on Numerical Implementation]
The sought-after $q$ in Lemma \ref{lem:ProjKLogSupDiff} can be determined with a one-dimensional Newton iteration.
The function $q \mapsto \partial_1 \theta(q^{-1/2},q^{1/2})$ becomes increasingly steep as $q \to 0$ which leads to increasingly unstable Newton iterations as $z_1$ approaches $0$. On $q \in [1,+\infty)$ the function is rather flat and easy to invert numerically.
To avoid these numerical problems, note that the roles of $z_1$ and $z_2$ in Lemma \ref{lem:ProjKLogSupDiff} can easily be swapped which corresponds to the transformation $q \leftrightarrow q^{-1}$.
Moreover, for $\max\{z_1,z_2\} < \frac{1}{2}$ one has $z \notin \supdiff \theta(0)$.
With this rule and by swapping the values of $z_1$ and $z_2$ if $z_1 < z_2$ one can always remain in the regime $q \in [1,+\infty)$.
Additionally, we recommend to replace the function $\theta(s,t)$ and its derivatives by a local Taylor expansion near the numerically unstable diagonal $s=t$.
\end{Remark}

\paragraph*{Geometric Mean.}
Furthermre, let us consider the case where $\theta=\theta_{\tn{geo}}$ is the geometric mean \eqref{def:means}.
For $s>0$ we again find $\lim_{t \searrow 0} \partial_1 \theta(t,s)= \lim_{t \searrow 0} \partial_2 \theta(s,t) = +\infty$ and consequently the same simplification of the algorithm applies as in the case of the logarithmic mean.
For the test of the inclusion $z=(z_1,z_2) \in \supdiff \theta(0)$, we argue as in the proof of Lemma \ref{lem:ProjKLogSupDiff}.
The functions $\partial_1 \theta(q^{-1/2},q^{1/2}) = \frac{1}{2} q^\frac12$  and  
$\partial_2 \theta(q^{-1/2},q^{1/2}) = \frac{1}{2} q^{-\frac12}$ have the same monotonicity properties as for the logarithmic mean.
Therefore, if $\min\{z_1,z_2\} \leq 0$ then $z \notin \supdiff \theta(0)$.
Otherwise, $q_1=4\,z_1^2$ and thus the condition $\partial_2 \theta(q_1^{-1/2},q_1^{1/2}) \leq z_2$ is equivalent to $z_1 \cdot z_2 \geq \frac14$.
To summarize, we have obtained
\begin{align*} 
\supdiff \theta(0) = \left\{ z\in \R^2 \,\colon\, 
z_1 \cdot z_2 \geq \tfrac{1}{4}  \wedge \min\{z_1,z_2\}>0 \right\}\,.
\end{align*}

\subsection[Proximal Mapping of I*J+/-]{Proximal Mapping of ${\indicatorFct^{\ast}_{\SetJRho}}$}
Note that $\indicatorFct_\SetJRho$ is a 1-homogeneous function. Hence, ${\indicatorFct^\ast_{\SetJRho}}$ will once again be an indicator function and $\prox_{\indicatorFct_\SetJRho^\ast}$ a projection. 
Consequently, the proximal mapping is independent of the step size $\sigma$, i.e.~$\prox_{\sigma\,\indicatorFct_\SetJRho^\ast}=\prox_{\indicatorFct_\SetJRho^\ast}$.
To compute $\prox_{\indicatorFct_\SetJRho^\ast}$ we use Moreau's decomposition \cite[Thm.~14.3]{ConvexFunctionalAnalysis-11} that implies
\begin{align}
\label{eq:ProjJDualMoreau}
\prox_{\indicatorFct_\SetJRho^\ast}=\id-\prox_{\indicatorFct_\SetJRho}=\id -\proj_\SetJRho
\end{align}
where $\id$ is the identity map on $\FESpaceZeroNodes \times (\FESpaceZeroEdges)^2$. To compute 
$\proj_\SetJRho(\massNodes, \massVX, \massVY )$ for a point 
$(\massNodes, \massVX, \massVY ) \in \FESpaceZeroNodes \times (\FESpaceZeroEdges)^2$ one has to find the minimizer 
$(\massNodes^\pr, {\massVX}^\pr, {\massVY}^\pr) \in \SetJRho$ of
\begin{align*}
\sum_{i=0}^{N-1} \NormV{\massNodes^\pr(t_i,\cdot)-\massNodes(t_i,\cdot)}^2 + \NormE{{\massVX}^\pr(t_i,\cdot,\cdot)-\massVX(t_i,\cdot,\cdot)}^2
+ \NormE{{\massVY}^\pr(t_i,\cdot,\cdot)-\massVY(t_i,\cdot,\cdot)}^2\,.
\end{align*}
Recall that for any $\massNodes^\pr \in \FESpaceZeroNodes$ there is precisely one pair $({\massVX}^\pr,{\massVY}^\pr) \in (\FESpaceZeroEdges)^2$ such that $(\massNodes^\pr,{\massVX}^\pr,{\massVY}^\pr) \in \SetJRho$, see \eqref{def:SetJ}. Therefore, one has to find $\massNodes^\pr \in \FESpaceZeroNodes$ which minimizes
\begin{align*}
&\sum_{i=0}^{N-1} \sum_{x \in \SetNodes} |\massNodes^\pr(t_i,x) - \massNodes(t_i,x)|^2 \pi(x)
 + \frac{1}{2}\!\! \sum_{(x,y) \in \SetNodes^2} |\massNodes^\pr(t_i,x) - \massVX(t_i,x,y)|^2 Q(x,y) \pi(x) \\
& + \frac{1}{2}\!\! \sum_{(x,y) \in \SetNodes^2} |\massNodes^\pr(t_i,y) - \massVY(t_i,x,y)|^2 Q(x,y) \pi(x)\,.
\end{align*}
The optimality condition in $\massNodes^\pr$ in combination with the reversibility $Q(x,y) \pi(x) = Q(y,x) \pi(y)$ yields for $i=0, \ldots, N-1$, $x \in \SetNodes$
\begin{align*}
\massNodes^\pr(t_i,x) = \frac{1}{2} \left( \massNodes(t_i,x) + \frac{1}{2} \sum_{y \in \SetNodes} (\massVX(t_i,x,y) + \massVY(t_i,y,x) ) Q(x,y) \right)
\end{align*}
and subsequently ${\massVX}^\pr(t_i,x,y)=\massNodes^\pr(t_i,x)$, ${\massVY}^\pr(t_i,x,y)=\massNodes^\pr(t_i,y)$ for $(x,y) \in \SetNodes \times \SetNodes$.
Finally, for  $(\massNodes^\pr,{\massVX}^\pr,{\massVY}^\pr)=\proj_{\SetJRho}(\massNodes,\massVX,\massVY)$ using \eqref{eq:ProjJDualMoreau} one gets 
$\prox_{\indicatorFct_{\SetJRho}^\ast}(\massNodes,\massVX,\massVY)=(\massNodes,\massVX,\massVY)-(\massNodes^\pr,{\massVX}^\pr,{\massVY}^\pr)$.

\subsection[Proximal Mapping of I*Javg]{Proximal Mapping of ${\indicatorFct^{\ast}_{\SetJavg}}$}
Once more, we use Moreau's decomposition, \eqref{eq:ProjJDualMoreau}, to compute the proximal mapping of $\indicatorFct^{\ast}_{\SetJavg}$ via the projection onto $\SetJavg$.
Note that the original problem \eqref{def:metricWithUAndV} does not change if we add the constraint $\massNodes_h(t_0,\cdot)=\massNodes_A$ and $\massNodes_h(t_N,\cdot)=\massNodes_B$ to the set $\SetJavg$. That is, we consider the projection onto the set
\begin{align*}
\SetJavgMod =\left\{ (\massNodes_h,\massElement_h \in \SetJavg \,\colon\,
\massNodes_h(t_0,\cdot)=\massNodes_A,\,
\massNodes_h(t_N,\cdot)=\massNodes_B \right\}.
\end{align*}
To compute the projection we have to solve
\begin{align*}
  \argmin_{ ( \massNodes_h^\pr, \massElement_h^\pr ) \in \SetJavgMod   } 
  \frac{1}{2} \sum_{i=0}^N  \sum_{x \in \SetNodes} |\massNodes_h^\pr(t_i,x) \!-\! \massNodes_h(t_i,x)|^2 \pi(x) 
  + \frac{1}{2} \sum_{i=0}^{N-1}  \sum_{x \in \SetNodes} |\massElement_h^\pr(t_i,x) \!-\! \massElement_h(t_i,x)|^2 \pi(x) \, . 
\end{align*}
Thus, we introduce a Lagrange multiplier $\lambda_h \in \FESpaceZeroNodes$ and define the corresponding Lagrangian
\begin{align*}
 L( \massNodes_h^\pr, \massElement_h^\pr, \lambda_h ) 
 & =  \frac{1}{2} \sum_{i=0}^N  \sum_{x \in \SetNodes} |\massNodes_h^\pr(t,x) \!-\! \massNodes_h(t,x)|^2 \pi(x) 
  + \frac{1}{2} \sum_{i=0}^{N-1}  \sum_{x \in \SetNodes} |\massElement_h^\pr(t,x) \!-\! \massElement_h(t,x)|^2 \pi(x) \\
 & - \sum_{i=0}^{N-1} \sum_{x \in \SetNodes} \lambda_h(t_i,x) \left( \avg_h \massNodes_h^\pr(t_i,x) - \massElement_h^\pr(t_i,x) \right) \pi(x)
 \, .
\end{align*}
We know directly from the added boundary constraints that
\begin{align*}
\massNodes_h^\pr(t_0,x) & = \massNodes_A, &
\massNodes_h^\pr(t_N,x) & = \massNodes_B.
\end{align*}
The optimality condition in $\massNodes_h^\pr$ for all $x \in \SetNodes$ and for interior time steps $i=1,\ldots,N-1$ reads as
\begin{align}\label{eq:AverageMassProjMassNodes}
 \massNodes_h^\pr(t_i,x) = \massNodes_h(t_i,x) + \tfrac{1}{2} (\lambda_h(t_{i-1},x) + \lambda_h(t_i,x) ) \, .
\end{align}
Further, the optimality condition in $\massElement_h^\pr$ implies that on each interval 
\begin{align}\label{eq:AverageMassProjMassElement}
 \massElement_h^\pr(t_i,x) = \massElement_h(t_i,x) - \lambda_h(t_i,x) \, .
\end{align}
Combining both with the constraint $\avg_h \massNodes_h^\pr(t_i,x) = \massElement_h^\pr(t_i,x)$, we obtain
\begin{align*}
 \massElement_h(t_i,x) - \lambda_h(t_i,x) 
 &= \massElement_h^\pr(t_i,x)
 = \avg_h \massNodes_h^\pr(t_i,x)\\
 &= \avg \massNodes_h(t_i,x) + \tfrac{1}{4} (\lambda_h(t_{i-1},x) + 2 \lambda_h(t_i,x) + \lambda_h(t_{i+1},x) ) 
\end{align*}
for all interior elements $\interval{i}$ with $i=1,\ldots,N-2$ and for all $x \in \SetNodes$.
Analogously, using the boundary conditions we get
\begin{align*}
 & \massElement_h(t_0,x) - \lambda_h(t_0,x) = \tfrac{1}{2} (\massNodes_A(x) + \massNodes_h(t_1,x) ) + \tfrac{1}{4} ( \lambda_h(t_0,x) + \lambda_h(t_1,x) ) \\
 & \massElement_h(t_{N-1},x) - \lambda_h(t_{N-1},x) = \tfrac{1}{2} (\massNodes_B(x) + \massNodes_h(t_{N-1},x) ) + \tfrac{1}{4} ( \lambda_h(t_{N-2},x) + \lambda_h(t_{N-1},x) ) \, .
\end{align*}
Thus, for each $x \in \SetNodes$ the Lagrange multiplier $\lambda_h$ satisfies the linear system of equations
\begin{align*}
 & \tfrac{1}{4} ( 5 \lambda_h(t_0,x) + \lambda_h(t_1,x) ) = \massElement_h(t_0,x) - \tfrac{1}{2} ( \massNodes_A(x) + \massNodes_h(t_1,x) ) \\
 & \tfrac{1}{4} ( \lambda_h(t_{i-1},x) + 6 \lambda_h(t_i,x) + \lambda_h(t_{i+1},x) ) = \massElement_h(t_i,x) - \tfrac{1}{2} ( \massNodes_h(t_{i+1},x) + \massNodes_h(t_i,x) ) \quad \forall i=1,\ldots,N-2 \\
 & \tfrac{1}{4} ( \lambda_h(t_{N-2},x) + 5 \lambda_h(t_{N-1},x) ) = \massElement_h(t_{N-1},x) - \tfrac{1}{2} ( \massNodes_B(x) + \massNodes_h(t_{N-1},x) )
\end{align*}
This system is solvable, since the corresponding matrix with diagonal $(5,6,\ldots,6,5)$ and off-diagonal $1$ is strictly diagonal dominant. 
Then, given the Lagrange multiplier $\lambda_h$, the solution of the projection problem is given by \eqref{eq:AverageMassProjMassNodes} and \eqref{eq:AverageMassProjMassElement}.
Finally, the proximal map of $\indicatorFct_{\SetJavg}^{\ast}$ can be computed by Moreau's identity, \eqref{eq:ProjJDualMoreau}.
Thus, to compute the proximal mapping of ${\indicatorFct^{\ast}_{\SetJavgMod}}$ one must solve a sparse system in time for each graph node separately. Since the involved matrix is constant, it can be pre-factored.

\subsection[Proximal Mapping of I*J=]{Proximal Mapping of $\indicatorFct_{\SetJequal}$}
The proximal map of $\indicatorFct_{\SetJequal}$ is  given by the projection
\begin{align*}
 \proj_{\SetJequal}( \massElement_h, \massElementSlack_h )
& = \argmin_{ ( \massElement_h^{\pr}, \massElementSlack_h^{\pr}) \in \FESpaceZeroEdges \times \FESpaceZeroEdges \; : \; \massElement_h^\pr = \massElementSlack_h^\pr } 
 \quad \frac{1}{2} h \sum_{i=0}^{N-1} \sum_{x,y \in \SetNodes}  \left( |\massElement_h - \massElement_h^\pr|^2 + |\massElementSlack_h - \massElementSlack_h^\pr|^2 \right) Q(x,y) \pi(x)\\
 &= \frac{1}{2} ( \massElement_h + \massElementSlack_h, \massElement_h + \massElementSlack_h ) \, .
\end{align*}

\section{Numerical Results}\label{sec:numerics}
In what follows we compare the numerical solution based on our
discretization with the explicitly known solution for a simple model
with just two nodes.  Furthermore, we apply our method to a set of
characteristic test cases to study the qualitative and quantitative
behaviour of the discrete transportation distance.
\paragraph*{Comparison with the exact solution for the 2-node case.}
Consider a two point graph $\SetNodes = \{a,b\}$ with Markov chain and stationary distribution
\begin{align*}
 Q = \begin{pmatrix}
      0 & p \\
      q & 0
     \end{pmatrix} \, ,
  \quad 
  \pi = \begin{pmatrix}
         \frac{q}{p+q} \\
         \frac{p}{p+q}
        \end{pmatrix} \, ,
\end{align*}
where $p,q \in (0,1]$.
For this case, Maas \cite{Ma11} constructed an explicit solution for the geodesic from 
$\massNodes_A = \left( \frac{p+q}{q}, 0 \right)$ to $\massNodes_B = \left( 0, \frac{p+q}{p} \right)$. 
Note that every probability measures on $\SetNodes$ can be described by a single parameter $r \in [-1,1]$ via
\begin{align*}
  \massNodes( r ) = ( \massNodes_a(r), \massNodes_b(r) ) = \left( \frac{p+q}{q} \frac{1-r}{2} \, , \frac{p+q}{p} \frac{1+r}{2} \right) .
\end{align*}
Especially, we have $\massNodes_A = \massNodes(-1)$ and $\massNodes_B = \massNodes(1)$.
Using this representation, Maas showed that for $-1 \leq \alpha \leq \beta \leq 1$
the optimal transport distance is given by
\begin{align}\label{eq:Distance2Point}
 \Wc( \massNodes(\alpha), \massNodes(\beta) ) = \frac{1}{2} \sqrt{\frac{1}{p} + \frac{1}{q}} \int_\alpha^\beta \frac{1}{\sqrt{\theta(\massNodes_a(r), \massNodes_b(r) )}} \d r
\end{align}
and the optimal transport geodesic from $\massNodes(\alpha)$ to $\massNodes(\beta)$ is given by $\massNodes(\gamma(t))$ for $t \in [0,1]$, 
where $\gamma$ satisfies the differential equation
\begin{align}\label{eq:ODE2Point}
 \gamma'(t) = 2 (\beta - \alpha) \Wc( \massNodes(\alpha), \massNodes(\beta) ) \sqrt{\frac{pq}{p+q} \theta( \massNodes(\gamma_a(t)), \massNodes_b(\gamma(t)) } \, .
\end{align}
For the special case, where $\theta$ is the logarithmic mean
$\theta_{\text{log}}$ and $p=q$, one obtains that
$\theta_{\text{log}} ( \massNodes_a(r), \massNodes_b(r) ) =
\frac{r}{\arctanh(r)} \, .$
and consequently the discrete transport distance is given by
$ \Wc( \massNodes(\alpha), \massNodes(\beta) ) = \frac{1}{\sqrt{2p}}
\int_\alpha^\beta \sqrt{ \frac{\arctanh(r)}{r} } \d r\,.  $
Furthermore, the optimal transport geodesic from $\massNodes(\alpha)$
to $\massNodes(\beta)$ is given by $\massNodes(\gamma(t))$ for
$t \in [0,1]$, where $\gamma$ satisfies the differential equation
$ \gamma'(t) = \sqrt{2p} (\beta - \alpha) \Wc( \massNodes(\alpha),
\massNodes(\beta) ) \sqrt{ \frac{\gamma(t)}{\arctanh(\gamma(t))} } \,
.  $
For this two point graph we numerically compute the optimal transport
geodesic. This allows us to evaluate directly the distance $\Wc$,
which we can compare with a numerical quadrature of
\eqref{eq:Distance2Point}.  Using the approximation of $\Wc$, we use
an explicit Euler scheme to compute the solution
$\massNodes_h^{\text{ODE}}$ of the ODE \eqref{eq:ODE2Point}.  For the
case $p=q=1$ we compare our numerical solution to the Euler
approximation for the ODE for $N=2000$ in Fig.~\ref{fig:testone}.
\begin{figure}

\centering
\begin{minipage}{0.8\textwidth}
\resizebox{1.0\textwidth}{!}{
\begin{tikzpicture}[x=\textwidth,y=\textwidth]
 \node at (0.0, 0){\includegraphics[scale=1.0]{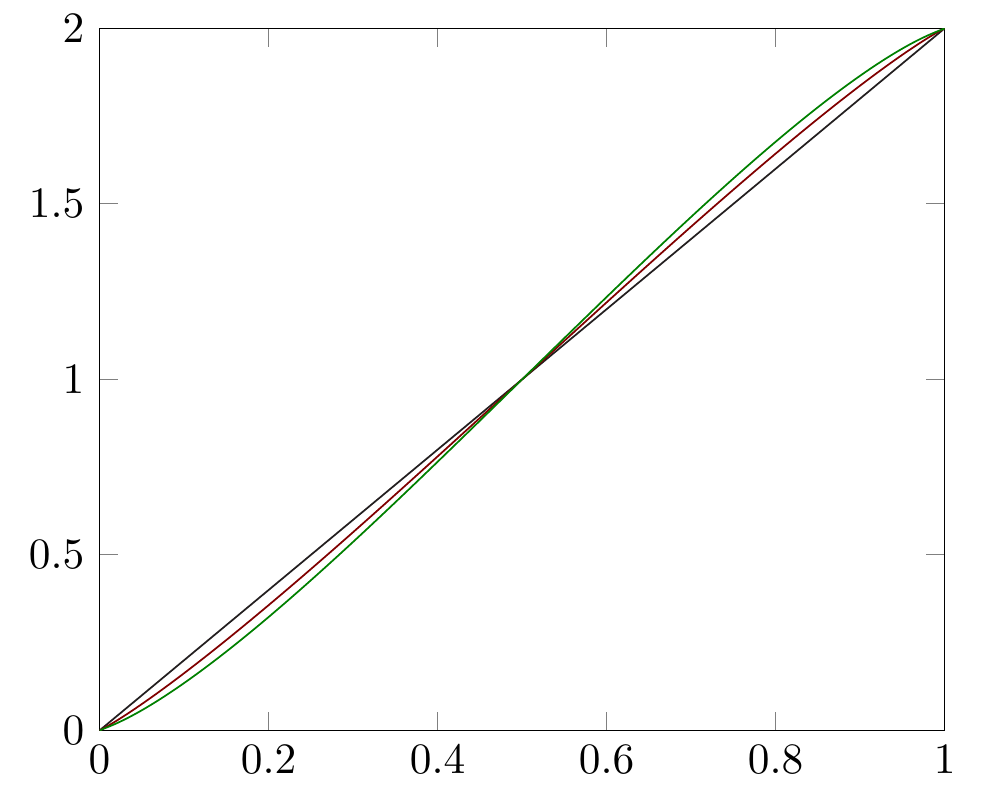}};
 \node at (1.0, 0){\includegraphics[scale=1.0]{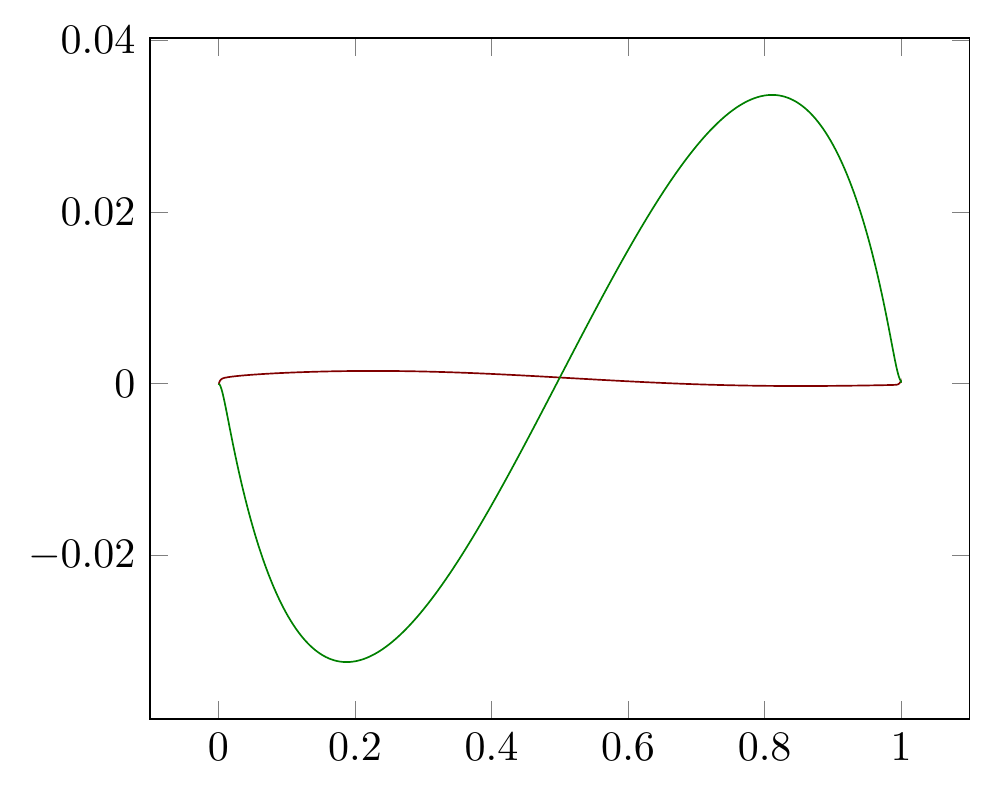}};
\end{tikzpicture}
}
\end{minipage}
\caption{The mass distribution at $b$ is plotted over time $t\in [0,1]$.
Left: Numerical solution for a 2-point graph $\SetNodes = \{a,b\}$ for the logarithmic (red) and geometric (green). The black line represents the diagonal, which 
is the solution in the case of the  (non admissible) arithmetic averaging. 
Right: Difference of the numerical solution for the logarithmic (red) and geometric (green) mean with the Euler scheme solution $\massNodes_h^{\text{ODE}}$ for the logarithmic mean. 
}
\label{fig:testone}
\end{figure}
\paragraph*{Geodesics on some selected graphs.}
Let us consider four different graphs whose nodes and edges form a triangle, the $3\times 3$ lattice, a cube, and a hypercube, respectively.
Figure~\ref{fig:labels} depicts these graphs with labeled nodes and edges. 
In all cases, we set for each node $x$ with $m$ outgoing edges $\pi(x) = \frac{m}{|E|}$ and $Q(x,y) = \frac{1}{\pi(x)|E|}$.  
Figure~\ref{fig:simple} shows numerically computed geodesic paths. The underlying time step size is $h=\frac1{100}$.
\begin{figure}
\centering
\resizebox{0.8\textwidth}{!}{
\begin{tikzpicture}[x=\textwidth,y=\textwidth]
 \node at (0, 0.45){\includegraphics[page=3, height=0.225\textwidth]{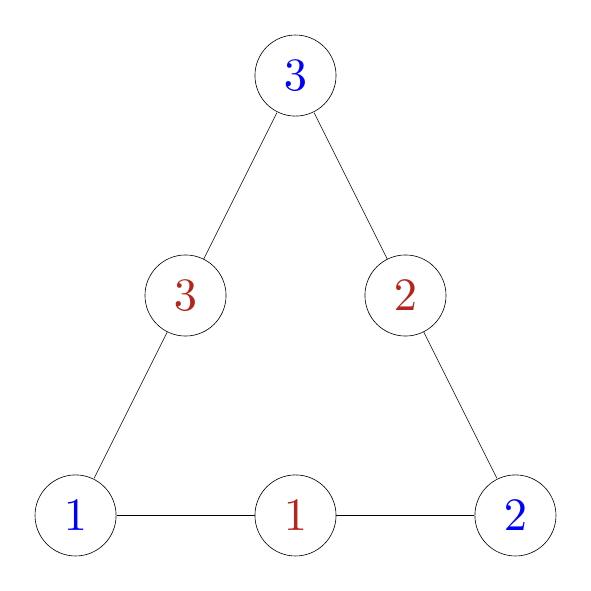}};
 \node at (0, 0.24){\includegraphics[page=3, height=0.225\textwidth]{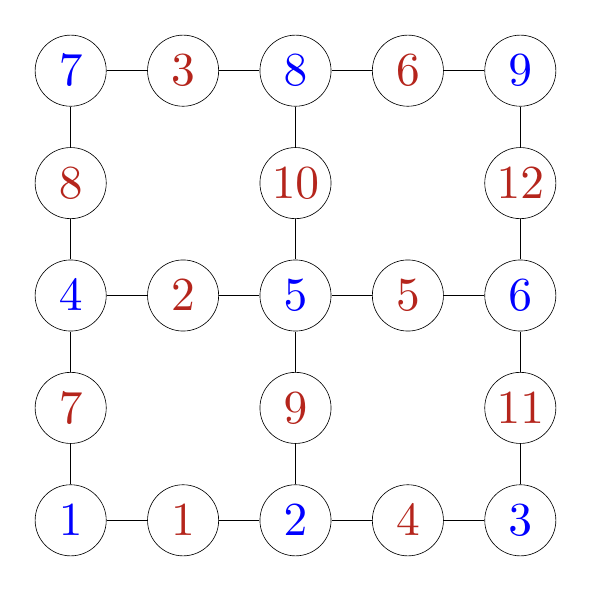}};
 \node at (0, 0){\includegraphics[page=3, height=0.3\textwidth]{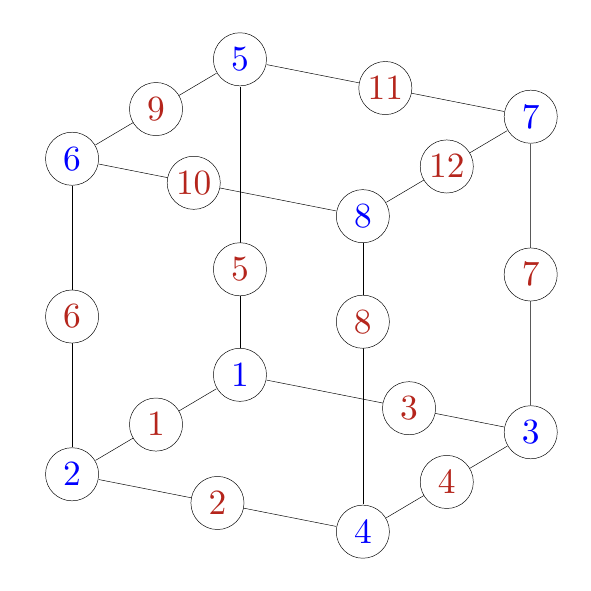}};
  \node at (0.45, 0.2){\includegraphics[page=3, height=0.6\textwidth]{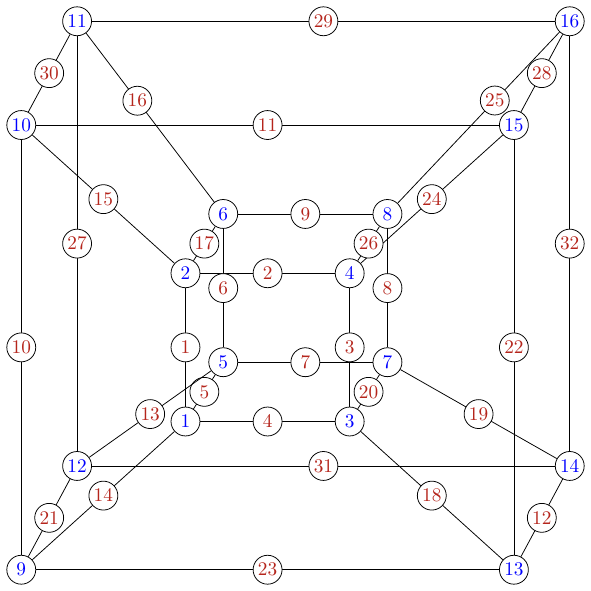}};
\end{tikzpicture}
}

\caption{Labeling of nodes and edges for four different graphs: a triangle, the $3x3$ lattice, a cube, and a hypercube.}
\label{fig:labels}
\end{figure}
The solution $(\massNodes, \momentum)$ is displayed at intermediate
time steps indicated on the arrow in the first row.  For each of these
time steps, blue discs and red arrows superimposed over the graph
display mass and momentum at nodes and on edges, respectively.  The
area of a disc is proportional to the mass $\massNodes(x) \pi(x)$.  A
red arrow connecting nodes $x$ and $y$ renders the momentum
$\momentum(x,y)$.  The direction of the arrow indicates the direction
of the flow, i.e.~it points from $x$ to $y$ if
$\momentum(x,y) = - \momentum(y,x) > 0$
(cf. Lemma~\ref{Lem:AntisymmetryM}).  The thickness of an arrow is
proportional to $|\momentum(x,y)|Q(x,y)\pi(x)$.  Underneath these
graph drawings both, mass and the momentum on nodes and edges, are
plotted in histograms.  The numbering of the columns in these plots
refers to the numbering of nodes and edges in Figure~\ref{fig:labels}.
The plots associated with $t=0$ and $t=1$ show the prescribed boundary
conditions in time.  As the stopping criteria for the iterative
algorithm in \eqref{eq:ChambollePock} we choose
$\int_0^1\NormV{\massNodes^{k+1} - \massNodes^k}^2\d t$ with threshold $10^{-10}$,
where $k$ denotes the iteration step.  Figure~\ref{fig:hypercube}
visualizes in the same fashion an optimal transport path on the graph
of the hypercube.  Note that for the cube, the hypercube, and the
$3\times 3$ lattice the computed solutions are symmetric. In explicit,
mass and momentum at time $t$ equal the mass and the momentum at time
$1-t$ on point reflected nodes and edges, respectively.  Furthermore,
for the cube and the hypercube the distribution of mass is constant on
all nodes at time $t=\frac12$.  Finally, in
Figure~\ref{fig:MChangesSign} we depict an example of graph with four
nodes, which shows that the sign of the momentum variable on a fixed
edge may change along a geodesic path.

\begin{figure}
\resizebox{0.9\textwidth}{!}{
\begin{tikzpicture}[x=\paperwidth,y=\paperwidth]
 \node at (0, 0){\includegraphics[page=1, scale=1]{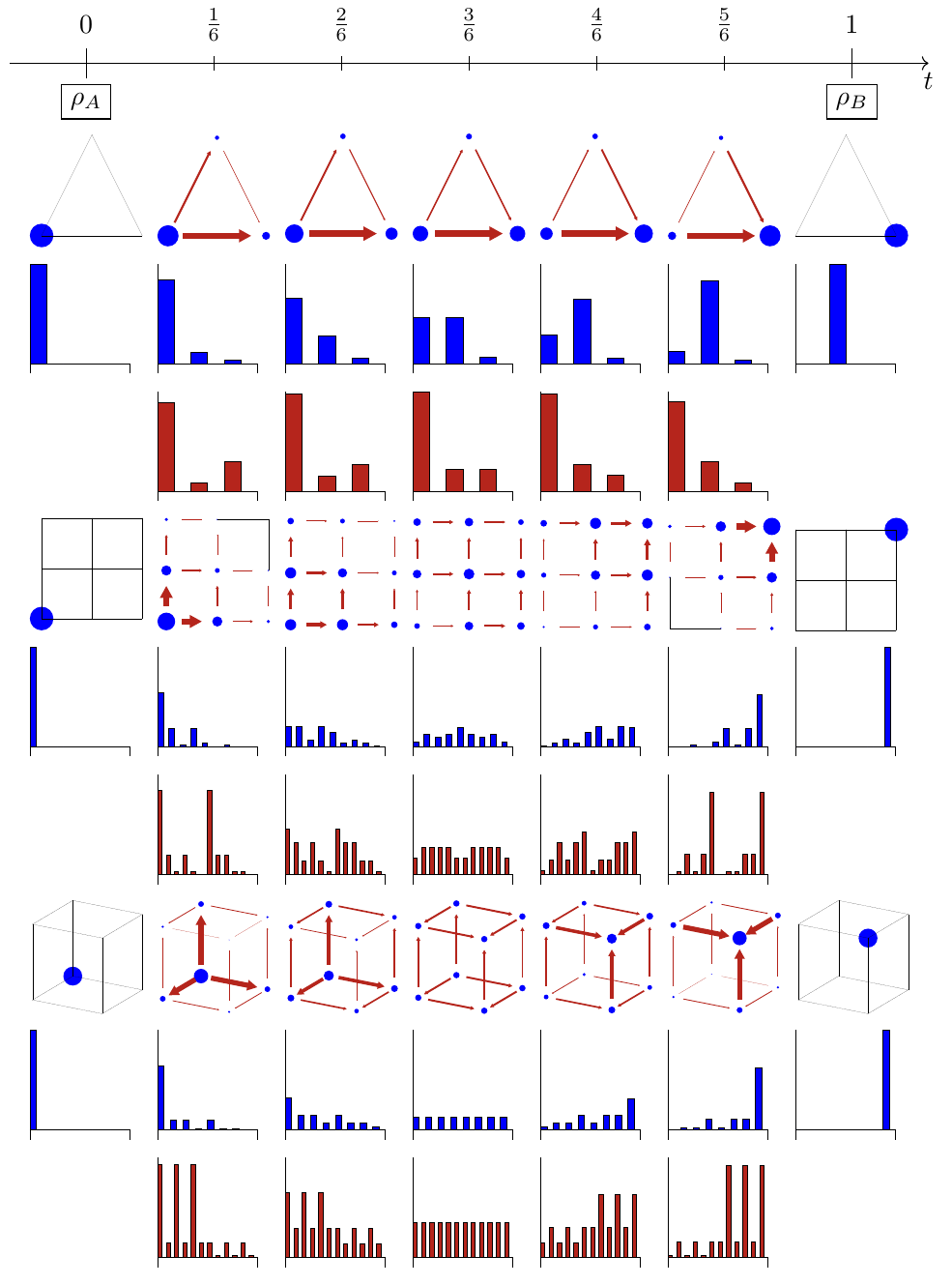}};
\end{tikzpicture}
}
\caption{Numerically computed geodesics on a triangle, a square
  lattice and a cube for prescribed boundary conditions at time $0$
  and $1$. Note in particular the symmetry under time reversal and the spreading of mass at intermediate
  times (equidistribution at $t=\frac12$ for the cube).}
\label{fig:simple}
\end{figure}

\begin{figure}
\centering
\resizebox{1.0\textwidth}{!}{
\begin{tikzpicture}[x=\paperwidth,y=\paperwidth]
 \node at (0, 0){\includegraphics[page=1, scale=0.8]{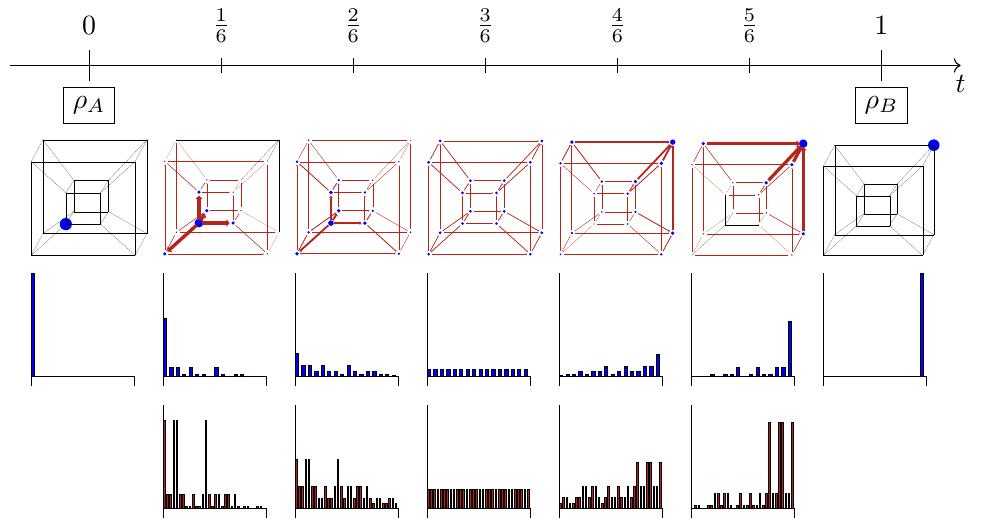}};
\end{tikzpicture}
}
\caption{Top: Numerically computed geodesic on a hypercube. Bottom: Distribution of mass and momentum, note again the symmetry under time reversal and the spreading of mass, with equidistribution at time $t=\tfrac12$.}
\label{fig:hypercube}
\end{figure}

\begin{figure}
\centering
\resizebox{1.0\textwidth}{!}{
\begin{tikzpicture}[x=\paperwidth,y=\paperwidth]
 \node at (0, 0){\includegraphics[page=1, scale=0.8]{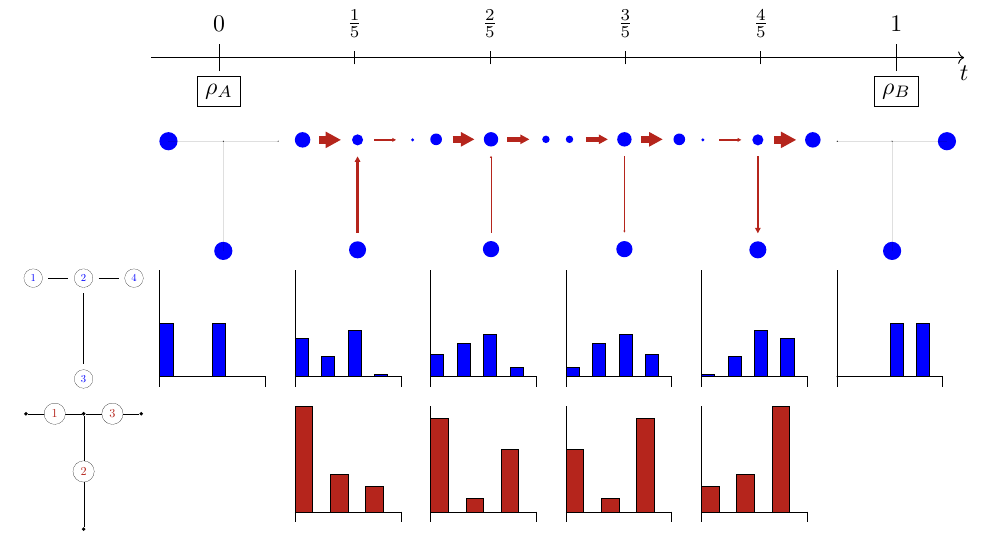}};
\end{tikzpicture}
}
\caption{Numerically computed geodesic on a graph with four nodes. Note that the sign of $\momentum$ for edge $2$ changes (cf.~$t=\frac{1}{5}$ and $t=\frac{4}{5}$).}
\label{fig:MChangesSign}
\end{figure}

\paragraph*{Experimental results related to the Gromov-Hausdorff convergence for simple graphs.} 
In \cite{GiMa13} it was shown that for the $d$-dimensional torus
$\mathbf{T}^d$ the discrete transportation distance $\Wc$ on a
discretized torus $\mathbf{T}^d_M$ with uniform mesh size
$\frac{1}{M}$ converges in the Gromov-Hausdorff metric to the
classical $L^2$-Wasserstein distance on $\mathbf{T}^d$.  In fact, the
optimal transport with respect to the classical $L^2$-Wasserstein
distance between two point masses is a point mass travelling along the
connecting straight line.  Concerning the expected concentration of
the transport along this line we perform the following numerical
experiments for $d=1,\,2$.  We first consider for $d=1$ the unit
interval $I = [0,1]$ and a sequence of space discretizations
$\SetNodes_M = \{ x_0, \ldots, x_{M} \}$ with uniform mesh size
$\frac1M$ with $M\in \N$.  The corresponding Markov kernel $Q_M$ for
$\SetNodes_M$ is defined by
$Q_M(x_i, x_{i+1}) = Q_M(x_i, x_{i-1}) = \frac{1}{2}$ for
$i=1,\ldots, x_{M-1}$ and $Q_M(x_0,x_1) = 1 = Q_M(x_{M},x_{M-1})$.
The continuous $L^2$-Wasserstein geodesic connecting
$\massNodes_A = \delta_0$ and $\massNodes_B = \delta_1$ is given by
the transport of the Dirac measure with constant speed:
\begin{align*}
 \massNodes(t,x) = \delta_t(x) \, .
\end{align*}
In Figure~\ref{fig:GromovHausdorff} we plot the density distribution of the discrete optimal transport geodesic at time $t=\frac{1}{2}$ for different grid sizes 
$\frac1M$. One observes the onset of mass concentration in space at that time at the location $x=\frac12$ for increasing $M$.
\begin{figure}
\centering
\resizebox{0.5\textwidth}{!}{
\begin{tikzpicture}[x=\textwidth,y=\textwidth]
 \node at (0, 0){\includegraphics[scale=3]{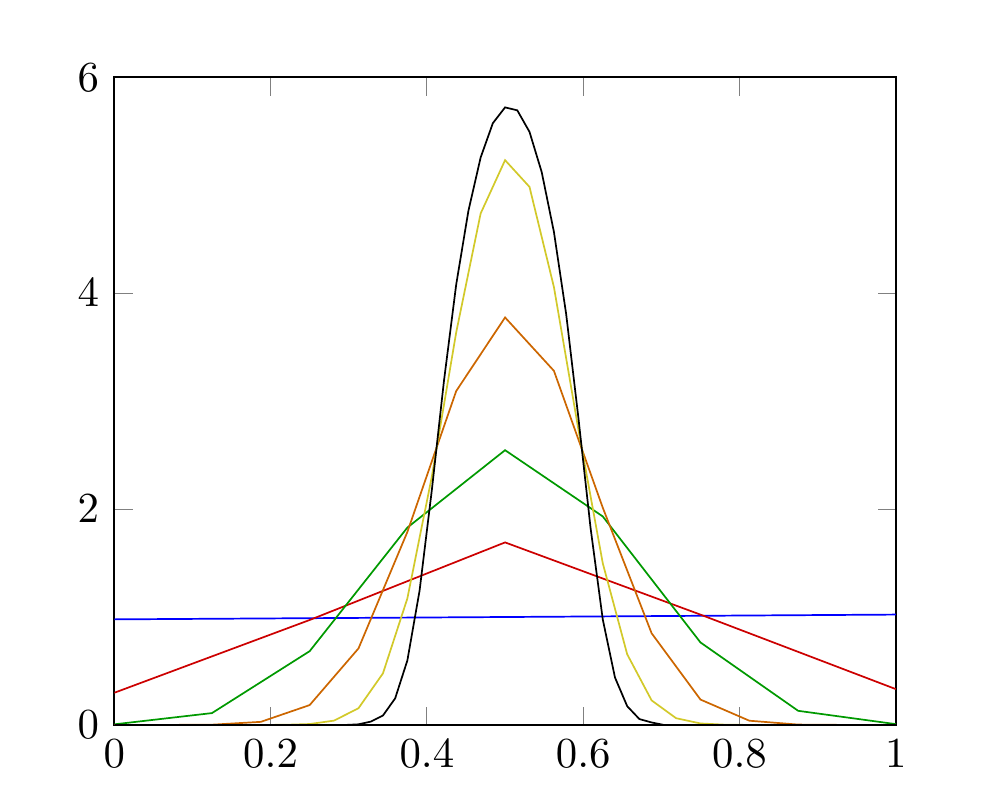}};
\end{tikzpicture}
}
\caption{Linearly interpolated densities for the $\Wc$ geodesic on a one dimensional chain graph between a Dirac mass at the beginning and the end, at $t=0.5$ with $M=2$ (blue), $4$ (red), $8$ (green), $16$ (orange), $32$ (yellow), and $64$ (black).}
\label{fig:GromovHausdorff}
\end{figure}
For $d=2$ we consider a square lattice of uniform grid size $\frac1M$ with $M\in \N$ and nodes $\SetNodes_M = \{ (i/M,j/M) \; : \; i,j \in (0,\ldots,M) \}$. 
The weights of the Markov kernel $Q$ are proportional to the number of adjacent edges.
Now, we investigate a discrete geodesic connecting the Dirac masses $\delta_{(0,0)}$ and $\delta_{(1,1)}$. One expects that for increasing $M$ mass on bands parallel to the space diagonal will decrease.  
In Figure \ref{fig:GromovHausdorffSquare} we plot for decreasing mesh size $\frac{1}{M}$ the in time accumulated density values along the diagonal and the off-diagonals bands of nodes.
More precisely, we define the bands of nodes $l^i_M = \{ (x_1,x_2) \in \SetNodes_M \times \SetNodes_M \; : \; x_2 = x_1 + \frac{i}{M} \}$ ($i=0$ being the diagonal) and compare the values $\int_0^1 \sum_{x \in l^i_M} \massNodes(t,x) \pi(x) \d t$.
\begin{figure}
\centering
\resizebox{0.9\textwidth}{!}{
\begin{tikzpicture}[x=\textwidth,y=\textwidth]
 \node at (0, 0){\includegraphics[scale=1]{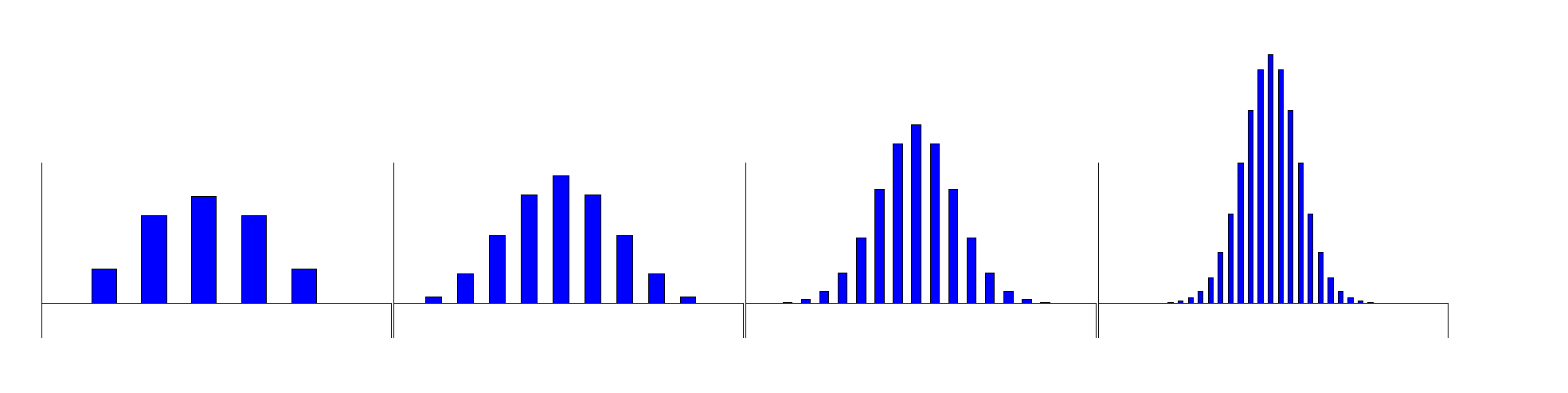}};
\end{tikzpicture}
}
\caption{
 Geodesics in the distance $\Wc$ on a two dimensional grid graph between Dirac masses at diagonally opposite ends. We show accumulated densities along the diagonal and the off-diagonals (see text for details).
From left to right: $M=4$, $8$, $16$, $32$. 
The width of the bars is scaled with the number of lines.
}
\label{fig:GromovHausdorffSquare}
\end{figure}
\paragraph*{Discrete geodesics on an internet network of Europe.}
In Figure~\ref{fig:Internet} we apply the investigated optimal
transport model to a coarse scale internet network of Europe and show
experimental results with masses (data packages) transported from
Dublin, Lisbon, and Madrid to Athens, Stockholm, and Kiev.  Also here,
we set for each node $x$ with $m$ outgoing edges
$\pi(x) = \frac{m}{|E|}$ and $Q(x,y) = \frac{1}{\pi(x)|E|}$, with
$|E|$ the total number of (directed) edges.

\begin{figure}
\resizebox{0.95\textwidth}{!}{
\begin{tikzpicture}[x=\paperwidth,y=\paperwidth]
 \node at (0, 0){\includegraphics[scale=0.6]{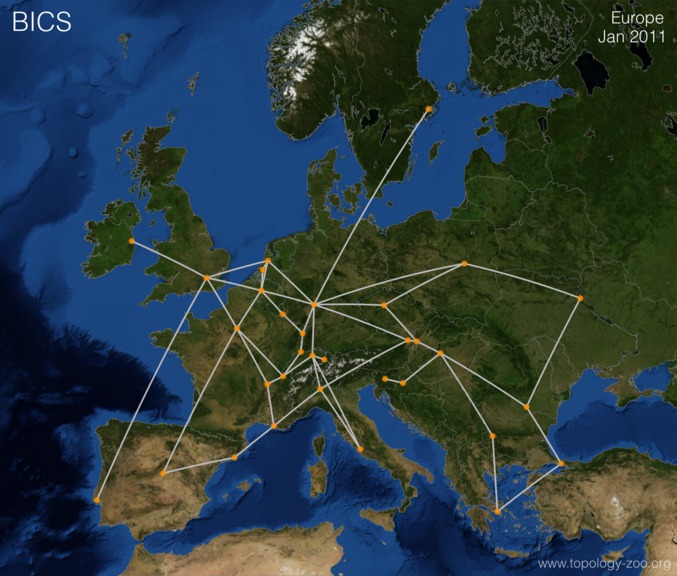}};
 \node at (0.3, 0){\includegraphics[page=1, scale=1]{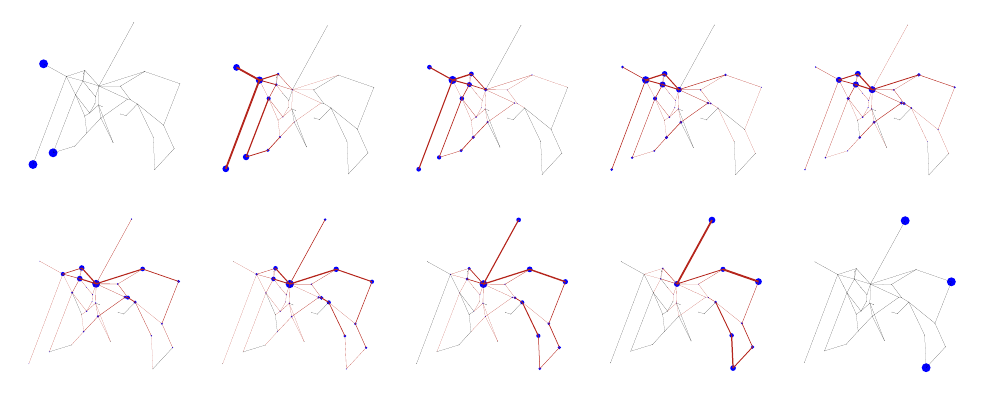}};
\end{tikzpicture}
}
\caption{Extraction of a discrete optimal transport geodesic.}
\label{fig:Internet}
\end{figure}

\section{Simulation of the gradient flow of the entropy}\label{sec:gradflow}

The entropy functional on $\Pc(\SetNodes)$ is given by
\begin{align*}
 \Entropy(\massNodes) = \sum_{x \in \SetNodes} \massNodes(x) \log(\massNodes(x)) \pi(x) \, .
\end{align*}
with the usual convention `$0\log0 = 0$'.  Maas \cite{Ma11} proved
that for the logarithmic mean $\theta_{\tn{log}}(\cdot,\cdot)$ and
$\massNodes \in \Pc(\SetNodes)$ the heat flow
$t \mapsto e^{t \triangle_\SetNodes} \massNodes$ is a gradient flow
trajectory for the entropy $\Entropy(\massNodes)$ with respect to the
discrete transportation distance $\Wc$.  In \cite{MaasDiscretePME2014}
it was shown that a similar result holds true for the Renyi entropy
\begin{align*}
 \Entropy_m(\massNodes) = \frac{1}{m-1} \sum_{x \in \SetNodes} \massNodes(x)^m \pi(x) \, . 
\end{align*}
In fact, for $m=\tfrac12$ and the gradient flow of $\Entropy_m$ with
respect to the metric $\Wc$ constructed with $\theta$ being the
geometric mean $\theta_{\tn{geom}}(\cdot,\cdot)$ is given by the
Fokker-Planck equation
$\partial_t \massNodes = \triangle_\SetNodes \massNodes^m$.

To verify this property numerically, we consider a line of five points with stationary distribution $\pi = \frac{1}{5}(1,2,2,2,1)$, Markov kernel $Q(x,y) = \frac{1}{10 \pi(x)}$ for $x,y$ adjacent, 
and initial mass $\massNodes = \frac{1}{10}(1,1,5,1,1)$.
Following \cite{JKO1998,AmGiSa08}, for an initial density $\massNodes_0 \in \Pc(\SetNodes)$ and a time step size $\tau > 0$ an implicit time-discrete gradient flow scheme for $\Entropy$ can be defined by
\begin{align}\label{eq:GradStep}
 \massNodes_{k+1} = \argmin_{\massNodes_B} \frac{1}{2} \W_h(\massNodes_{k},\massNodes_B)^2 + \tau \cdot  \Entropy(\massNodes_B)
\end{align}
with an inner time step size $h$ appearing in the discretization $\W_h$ of $\W$.
To minimize this functional numerically, we simultaneously carry out the external optimization over $\massNodes$ and the internal optimization within $\W_h$ . 
To this end, we define a discrete continuity equation with one free endpoint. For initial datum $\massNodes_A \in \Pc(\SetNodes)$ let
\begin{align*}
 \CE_h(\massNodes_A) = \left\{ (\massNodes_h, \momentum_h, \massNodes_B) \in  \FESpaceOneNodes \times \FESpaceZeroEdges \times \R^{\SetNodes} \,:\, (\massNodes_h,\momentum_h) \in \CE_h(\massNodes_A,\massNodes_B) \right\}\,.
\end{align*}
Analogous to \eqref{def:metricWithUAndV}, problem \eqref{eq:GradStep} can be written as
\begin{multline*}
 \min
 \; \{ \F(\massNodes_h,\momentum_h,\massEdges_h,\massVX_h,\massVY_h,\massElement_h, \massElementSlack_h,\massNodes_B) + 
 	\G(\massNodes_h,\momentum_h,\massEdges_h,\massVX_h,\massVY_h,\massElement_h, \massElementSlack_h,\massNodes_B)
 \,:\,  \\
 (\massNodes_h,\momentum_h,\massEdges_h,\massVX_h,\massVY_h,\massElement_h, \massElementSlack_h,\massNodes_B) \in \FESpaceOneNodes \times (\FESpaceZeroEdges)^4 \times (\FESpaceZeroNodes)^2 \times \R^{\SetNodes}
 \} 
\end{multline*}
with
\begin{align*}
 \F(\massNodes_h,\momentum_h,\massEdges_h,\massVX_h,\massVY_h,\massElement_h, \massElementSlack_h,\massNodes_B)  := &
  \ActionU(\massEdges_h, \momentum_h) 
  + \indicatorFct_{\SetJRho}(\massElementSlack_h,\massVX_h, \massVY_h)
  + \indicatorFct_{\SetJavg}(\massNodes_h,\massElement_h)
 + 2\,\tau \cdot \Entropy(\massNodes_B)\,, \\
  \G(\massNodes_h,\momentum_h,\massEdges_h,\massVX_h,\massVY_h,\massElement_h, \massElementSlack_h,\massNodes_B)   := &
  \indicatorFct_{\CE_h(\massNodes_k)}(\massNodes_h, \momentum_h, \massNodes_B)
  + \indicatorFct_{\SetK}(\massVX_h,\massVY_h,\massEdges_h )
  + \indicatorFct_{\SetJequal}(\massElement_h,\massElementSlack_h)\,.
\end{align*}
Again, this is amenable for algorithm \eqref{eq:ChambollePock}. 
We extend the space $H$ by a factor $\R^{\SetNodes}$ and adapt the scalar product on $H$ \eqref{eq:ScalarProductDiscrete} adding the term $h\,\InProdV{\massNodes_{B,1}(\cdot)}{\massNodes_{B,2}(\cdot)}$ with respect or the additional variable $\massNodes_B$. 
The proximal step of $\F^\ast$ then entails an additional proximal step of $(2\,\tau \cdot \Entropy)^\ast$ with respect to $h \NormV{\cdot}$
and in the proximal step of $G$ the projection onto $\CE_h(\massNodes_A,\massNodes_B)$ is replaced by a projection onto $\CE_h(\massNodes_k)$.
Next, we detail these modifications.

\newcommand{\EnergyWeight}{\gamma}
Let us recall that the proximal mapping of $(\EnergyWeight \cdot \Entropy)^\ast$ and $\EnergyWeight \cdot \Entropy$ are linked by Moreau's decomposition, cf.~\eqref{eq:ProjJDualMoreau}.
The computation of the the proximal mapping for $\EnergyWeight \cdot  \Entropy$ decouples in space and the resulting one dimensional problem can be solved via Newton's method.
This decoupling is possible since we do not enforce the constraint $\massNodes_B \in \Pc(\SetNodes)$ in the formulation of $\Entropy$ but enforce it via the discrete continuity equation constraint.

To implement the projection 
\begin{multline}
\proj_{\CE_h(\massNodes_A)}(\massNodes,\momentum,\massNodes_B) = 
  \argmin_{(\massNodes^\pr,\momentum^\pr,\massNodes^\pr_B) \in \CE_h(\massNodes_A)}
\frac{h}{2} \sum_{i=0}^N \NormV{\massNodes^\pr_h(t_i,\cdot)-\massNodes_h(t_i,\cdot)}^2 \\
+ \frac{h}{2} \sum_{i=0}^{N-1} \NormE{\momentum^\pr_h(t_i,\cdot)-\momentum_h(t_i,\cdot)}^2
+ \frac{h}{2} \NormV{\massNodes^\pr_B-\massNodes_B}^2
\label{eq:ProjCEFree}
\end{multline}
onto the set $\CE_h(\massNodes_A)$ of solutions of the discrete continuity equation with initial data $\massNodes_A$
the following modifications apply.
Analogous to Proposition \ref{prop:ProjCEDual}, a space time discrete elliptic equation
\begin{align*}\label{eq:ProjCELMFree}
 &\frac{ \LagrangeMultiplierCE_h(t_1,x) - \LagrangeMultiplierCE_h(t_{0},x) }{h^2} + \triangle_\SetNodes \LagrangeMultiplierCE_h(t_0,x)
  =  - \left( \frac{\massNodes_h(t_{1},x) - \massNodes_A(x)}{h} + \div \momentum_h(t_0,x) \right) \; , \\
 &\frac{ - \tfrac{3}{2} \LagrangeMultiplierCE_h(t_{N-1},x) - \LagrangeMultiplierCE_h(t_{N-2},x)}{h^2} + \triangle_\SetNodes \LagrangeMultiplierCE_h(t_{N-1},x)  \\
& \; \qquad\qquad\qquad\qquad \qquad\qquad\qquad = - \left( ( \frac{ \tfrac{1}{2}(\massNodes_B(x) + \massNodes_h(t_N,x)) - \massNodes_h(t_{N-1},x)}{h} + \div \momentum_h(t_{N-1},x) \right) \; ,\\
& \frac{\LagrangeMultiplierCE_h(t_{i+1},x) - 2 \LagrangeMultiplierCE_h(t_i,x) + \LagrangeMultiplierCE_h(t_{i-1},x) }{h^2} + \triangle_\SetNodes \LagrangeMultiplierCE_h(t_i,x)\\
&  \; \qquad\qquad\qquad\qquad \qquad\qquad\qquad  = - \left( \frac{\massNodes_h(t_{i+1},x) - \massNodes_h(t_i,x)}{h} + \div \momentum_h(t_i,x) \right)  \\
\end{align*}
with $i=1,\ldots,N-2$ and $x \in \SetNodes$
has to be solved for the Lagrange multiplier $\LagrangeMultiplierCE_h \in \FESpaceZeroNodes$.
Note that this system is no longer degenerate due to the additional freedom of $\massNodes_B$ and thus no regularization as before is required.
Then the solution $(\massNodes^\pr,\momentum^\pr,\massNodes^\pr_B)$ to \eqref{eq:ProjCEFree} is given by
\begin{align*}
 \massNodes_B^\pr(x) &= \frac{1}{2} \left( \massNodes_h(t_N,x) + \massNodes_B(x) - \frac{\LagrangeMultiplierCE_h(t_{N-1},x)}{h}  \right) \, , \\
 \massNodes^\pr_h(t_i,x)  & = \massNodes_h(t_i,x) + \frac{\LagrangeMultiplierCE_h(t_i,x) - \LagrangeMultiplierCE_h(t_{i-1},x)}{h}   \, , \\
 \massNodes^\pr_h(t_0,x)  & = \massNodes_A(x) \, , \; \massNodes^\pr_h(t_N,x) = \massNodes_B^\pr(x) \, , \\
 \momentum^\pr_h(t_i,x,y) &= \momentum_h(t_i,x,y) + \nabla_\SetNodes \LagrangeMultiplierCE_h(t_i,x,y)    
\end{align*}
for all  $i=1,\ldots,N-2$ and $x,y \in \SetNodes$.
\begin{figure}
\resizebox{1.0\textwidth}{!}{
\begin{tikzpicture}[x=\paperwidth,y=0.3\paperwidth]
\node at (0, 0){\includegraphics[page=1, scale=1]{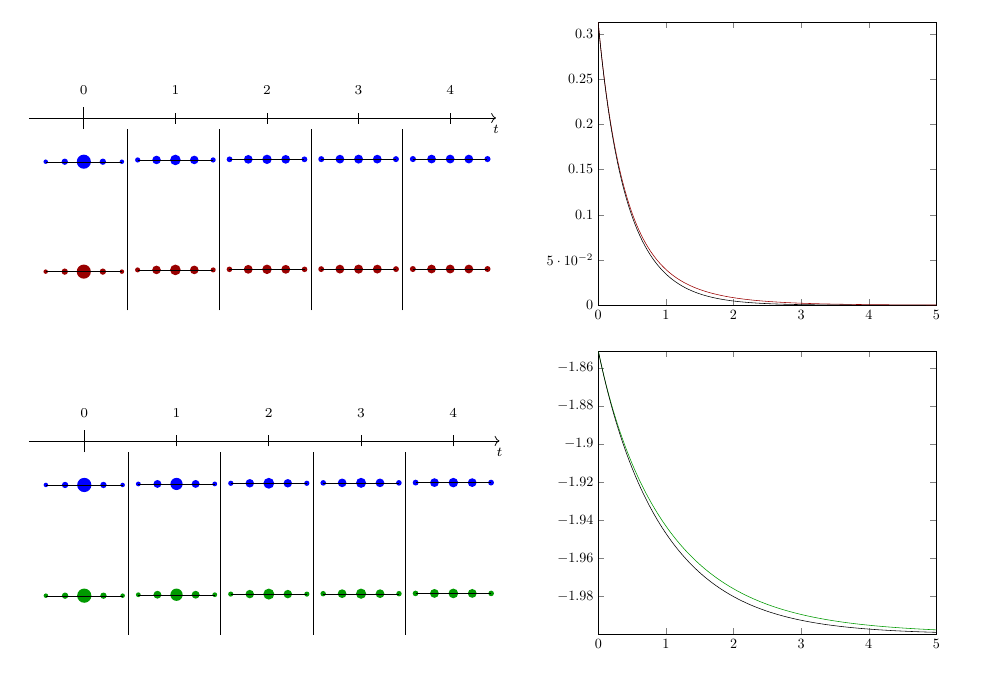}};
\end{tikzpicture}
}
\caption{Numerical solution of the heat flow (top) and  the Fokker-Planck equation (bottom) based on an explicit Euler scheme (blue) with time step size $10^{-3}$
and for the gradient flow of the associated entropy using the logarithmic mean (red) and the geometric mean (green), respectively, with $\tau=10^{-3}$ and $h=100$.
Panels on the left show the mass distributions on the graph at different times, panels on the right show the values of the entropies over time. }
\label{fig:heatLine}
\end{figure}
In Figure~\ref{fig:heatLine} we compare the numerical results for this natural discretization of the gradient flow of the entropy to 
the flow computed numerically with a simple explicit Euler discretization applied to the heat equation and the Fokker-Planck equation, respectively, with respect to the underlying Markov kernel.

\bibliographystyle{alpha}
\bibliography{all,own,library,local}

\end{document}